\documentclass[a4paper]{amsart}

\usepackage{xcolor}
\usepackage[
  pdftitle={Algebraic singular functions are not always dense in the ideal of $C^*$-singular functions},
  pdfauthor={Diego Mart\'{i}nez, N\'ora Szak\'acs},
  pdfsubject={Mathematics},
  colorlinks=true, linkcolor=blue, linkbordercolor=blue, citecolor=red, citebordercolor=red, urlcolor=blue, linktocpage=true
]{hyperref}
\usepackage[capitalise, nameinlink, noabbrev, nosort]{cleveref} 
\usepackage[lite]{amsrefs}

\usepackage[utf8]{inputenc}
\usepackage[T1]{fontenc}
\usepackage{lmodern, enumerate}
\usepackage{amssymb,amsxtra,amscd,amsmath,amsthm}
\usepackage{amsthm}
\usepackage{tikz-cd}
\usepackage{amsfonts}
\usepackage{comment}
\usepackage{tikz-cd}
\usepackage{bbm}
\usepackage[all]{xy}

\usepackage{nicefrac,mathtools}
\usepackage{physics}

\newcounter{thmintrocnt}
\numberwithin{equation}{section}
\theoremstyle{plain}
\newtheorem{theorem}[equation]{Theorem}
\crefname{theorem}{Theorem}{Theorems}
\newtheorem*{maintheorem}{Main Theorem}

\newtheorem{lemma}[equation]{Lemma}
\crefname{lemma}{Lemma}{Lemmas}
\newtheorem{proposition}[equation]{Proposition}
\crefname{proposition}{Proposition}{Propositions}

\crefname{claim}{Claim}{Claims}

\crefname{corollary}{Corollary}{Corollaries}
\theoremstyle{definition}

\crefname{definition}{Definition}{Definitions}
\theoremstyle{remark}

\crefname{convention}{Convention}{Conventions}
\newtheorem{remark}[equation]{Remark}
\crefname{remark}{Remark}{Remarks}

\crefname{notation}{Notation}{Notations}

\crefname{example}{Example}{Examples}

\crefname{question}{Question}{Questions}

\theoremstyle{plain}

\calclayout

\newcommand*{\MRref}[2]{\href{https://urldefense.com/v3/__http://www.ams.org/mathscinet-getitem?mr=*1*7D*7BMR__;IyUl!!PDiH4ENfjr2_Jw!C47ExQkO68EbYXdtff8bBZJdjVCYroxIJFsr-LDsshlAaC0WsKnfaYMjlF5Ltq-IKJekDjH_IQSpLcNcs5owNtiF_mJ1LpmqRDOjPlE$ [ams[.]org] #1}}
\newcommand*{\arxiv}[1]{\href{https://urldefense.com/v3/__http://www.arxiv.org/abs/*1*7D*7BarXiv__;IyUl!!PDiH4ENfjr2_Jw!C47ExQkO68EbYXdtff8bBZJdjVCYroxIJFsr-LDsshlAaC0WsKnfaYMjlF5Ltq-IKJekDjH_IQSpLcNcs5owNtiF_mJ1LpmqRL2_0so$ [arxiv[.]org]: #1}}

\newcommand*{\CC}{\mathbb{C}}             
             \newcommand*{\FF}{\mathbb{F}}

             \newcommand*{\NN}{\mathbb{N}}

             \newcommand*{\ZZ}{\mathbb{Z}}

            \newcommand*{\CB}{\mathcal{B}}
\newcommand*{\CCC}{\mathcal{C}}           
            
\newcommand*{\CG}{\mathcal{G}}

\newcommand*{\CU}{\mathcal{U}}

\newcommand*{\fa}{\mathfrak{a}}            \newcommand*{\fb}{\mathfrak{b}}
            
            \newcommand*{\ff}{\mathfrak{f}}
\newcommand*{\fg}{\mathfrak{g}}

\newcommand*{\cstar}{\texorpdfstring{\(C^*\)\nobreakdash-\hspace{0pt}}{*-}}
\newcommand*{\Star}{\texorpdfstring{\(^*\)\nobreakdash-\hspace{0pt}}{*-}}

\newcommand*{\supp}[1]{\text{supp}\left(#1\right)}

\newcommand*{\red}{\rm red}
\newcommand*{\full}{\rm max}
\newcommand*{\ess}{\rm ess}
\newcommand*{\alg}{\rm alg}

\newcommand*{\triv}{\rm triv}

\newcommand*{\pitriv}{\pi_{\triv}}

\newcommand*{\grpd}{\CG}
\newcommand*{\rest}{r}

\newcommand*{\rednorm}[1]{\norm{#1}_{\red}}

\newcommand*{\supnorm}[1]{\norm{#1}_{\infty}}

\newcommand*{\cont}[1]{C(#1)}
\newcommand*{\contc}[1]{C_c(#1)}
\newcommand*{\contz}[1]{C_0(#1)}

\newcommand*{\redalg}[1]{C^*_{\red}(#1)}              
\newcommand*{\maxalg}[1]{C^*_{\full}(#1)}             
\newcommand*{\essalg}[1]{C^*_{\ess}(#1)}              
\newcommand*{\algalg}[1]{\CCC_c(#1)}

\newcommand*{\steinalg}[1]{\CC(#1)}
\newcommand*{\stein}[1]{\CC(#1)}

\newcommand*{\singideal}{J}
\newcommand*{\asingideal}{\singideal_{\alg}}

\newcommand*{\borelb}[1]{B_b(#1)}

\newcommand{\ol}{\overline}
\newcommand*{\spn}{\operatorname{span}}
\newcommand{\normal}{\vartriangleleft}
\newcommand{\id}{\operatorname{id}}

\newcommand{\editB}[1]{{\color{black}#1}}

\begin{document}

\title[Algebraic singular functions are not always dense in the $C^*$-singular ideal]{Algebraic singular functions are not always dense in the ideal of $C^*$-singular functions}

\author[Diego Mart\'{i}nez]{Diego Mart\'{i}nez}
\address{Department of Mathematics, KU Leuven, Celestijnenlaan 200B, 3001 Leuven, Belgium.}
\email{diego.martinez@kuleuven.be}

\author[N\'ora Szak\'acs]{N\'ora Szak\'acs}
\address{Department of Mathematics, University of Manchester, Manchester M13 9PL, United Kingdom.}
\email{nora.szakacs@manchester.ac.uk}

\begin{abstract}
  We give the first examples of \'etale (non-Hausdorff) groupoids $\grpd$ whose \cstar{}algebras contain singular elements that cannot be approximated by singular elements in \(\algalg \grpd\).
  We provide two examples: one is a bundle of groups, and the other a minimal and effective groupoid constructed from a self-similar action on an infinite alphabet. Moreover, we also prove that the Baum--Connes assembly map for the first example is not surjective, not even on the level of its essential \cstar{}algebra.
\end{abstract}



\thanks{The first-named author was funded by project G085020N funded by the Research Foundation Flanders (FWO). Both authors would like to thank the Isaac Newton Institute for Mathematical Sciences, Cambridge, for support and hospitality during the programme Topological groupoids and their C*-algebras, where work on this paper was undertaken. This work was supported by EPSRC grant EP/V521929/1.}

\maketitle


\section{Introduction} \label[section]{sec:intro}

\cstar{}algebras associated to topological (particularly \'etale) groupoids play an important role in operator algebras. On the one hand, they constitute a broad class that encompasses a wide array of important examples \cites{Renault2008CartanSI,ExelBigPaper,li-classifiable-cartans-2020, neshveyevetal2023} and, on the other, they remain tractable through common techniques \cites{Renault80, SimsNotes2020, ExelPitts}.

Although \cstar{}algebras of non-Hausdorff étale groupoids were introduced by Connes in the early 80s, up to the late 2010s, most research focused on the \emph{Hausdorff} setting \cite{SimsNotes2020}, which is in many ways easier to understand. Nevertheless, non-Hausdorff groupoids arise naturally in important examples, notably in the context of dynamical systems, (\emph{e.g.}\ \cites{NekCrelle, ExelPardoSelfSim, Whittaker:ssgrpd}), which has motivated substantial progress in their study over the past decade \cites{CEPSS-2019,bkmk-2024-twited-grpds, buss-martinez-2025, Timmermann:2011:fell-compactification, BussMartinez:Approx-Prop, KwasniewskiMeyer-essential-cross-2021, Exel:nonHDgrpds}.

The main difficulty in the non-Hausdorff setting is that the groupoid \cstar{}algebra contains non-continuous functions. 
This is in particular apparent in the characterization of simplicity of these \cstar{}algebras. In the Hausdorff case, simplicity is ensured by dynamical properties of the groupoid called \emph{minimality} and \emph{effectiveness} \cite{brownetal-simplicity}, which are typically easy to verify on concrete examples. In the non-Hausdorff setting, one needs to impose the additional condition that the ($C^*$)-\emph{singular ideal}\footnote{\, The \cstar{} prefix is added here to distinguish it from its purely algebraic counterpart.} $\singideal$ is zero \cite{CEPSS-2019}, which is not implied by minimality and effectiveness \cite{BenNora21}. This ideal $\singideal$ consists of those functions that vanish on a dense subset (called \emph{singular functions}), and it is often difficult to concretely describe, or even to check if $J=0$.

The quotient of the groupoid \cstar{}algebra by the singular ideal $J$ is called the \emph{essential \cstar{}algebra}, introduced in \cites{ExelPitts,KwasniewskiMeyer-essential-cross-2021}. 
It is now widely accepted to be the groupoid \cstar{}algebra with the ``correct'' ideal structure \cites{CEPSS-2019, KennedyKimLiRaumUrsu2021}, and is the focus of a rapidly expanding body of research. Understanding the singular ideal thus becomes relevant and important both in the characterization of simplicity of groupoid \cstar{}algebras as well as in the study of essential \cstar{}algebras.

Much of the research on étale groupoid \cstar{}algebras goes hand in hand with the study of their discrete analogues introduced by Steinberg in \cite{BenSteinAlgPaper}. If $\grpd$ is an \emph{ample} groupoid, then its Steinberg algebra $\steinalg\grpd$ sits densely in the groupoid \cstar{}algebra $\redalg\grpd$ and often mirrors its algebraic properties. In particular, if $\grpd$ is Hausdorff and amenable, then the simplicity of $\steinalg\grpd$ and that of $\redalg\grpd$ coincide. In the non-Hausdorff case, \editB{the new obstruction to the simplicity of  $\steinalg\grpd$ is the so-called \emph{(algebraic) singular ideal}, defined as $J \cap \steinalg\grpd$. This naturally brings about the question of what the relationship between $J \cap \steinalg\grpd$ and $J$ is (beyond the obvious inclusion $\overline{J \cap \steinalg\grpd} \subseteq J$), which was first raised in \cite{CEPSS-2019}:}


\begin{enumerate}
  \item[Q1.] Does $J \cap \steinalg\grpd=0$ imply $J=0$ (at least in the minimal and effective case)?
  \item[Q2.] Is $J \cap \steinalg\grpd$ dense in the \cstar{}singular ideal $J$ (at least in the minimal and effective case)?
\end{enumerate}
These questions are especially interesting as they determine to what extent one can hope to understand  $J$ by studying $J \cap \steinalg\grpd$. The latter is often easier, not just \emph{a priori} but also in practice: for instance, when it comes to characterizing the simplicity of concrete non-Hausdorff groupoid \cstar{}algebras, the most successful approach has been to characterize when $J \cap \steinalg\grpd=0$, and showing it implied $J=0$ in that concrete setting \cites{BenNewSelfsim, Josiah}.

Until this point, all progress on the above questions was pointing to a positive answer. The first results focused on groupoids arising from self-similar group actions as defined in \cite{Nekrashevych2009}: \cite{Yosh} implied Q1 holds in the case of so-called \emph{multispinal groups}, which \cite{Nadiaetal} reproved for the subclass of $\ZZ_2$-multispinal groups using different techniques. In \cite{BenNewSelfsim}, this is extended to all \emph{contracting} self-similar groups.

The first significant step towards the general case was \cite{hausdorff-covers-2025}*{Thm. C} which shows that Q1 holds in a class of groupoids whose `non-Hausdorffness' is finitary in the sense that every net has (at most) a finite number of limits. This is extended by \cite{Hume2025}, which reduces Q1 to the isotropy group $C^*$-algebras and formulates a sufficient condition on these. 
\editB{Since the publication of our preprint, a new paper \cite{GH2025} has been published, which shows the stronger property Q2 for a large class of groupoids, including those with finitary `non-Hausdorffness', as well as second-countable amenable bundles of groups.}

In this paper, we present the first negative result in this direction of research by answering Q2 negatively. 

\begin{maintheorem} 
  There exist \'etale (in fact, ample) groupoids \(\grpd\) such that \(J \cap \algalg{\grpd}\) is not dense in \(\singideal\), and thus in particular $J \cap \steinalg\grpd$ is not dense in $J$. 
\end{maintheorem}

We construct two (classes of) examples showcasing the pathological behaviour above: the first one, presented in \cref{sec:bundle-groups}, is a (non-amenable) bundle of groups, and the second in \cref{sec:minimal-effective} is a minimal and effective groupoid associated to a self-similar group acting on an infinite alphabet. These constructions can be read independently from each other, but both follow the same basic strategy, which we describe in \cref{sec:strategy}. 

We remark that these examples do have algebraic singular functions in the Steinberg algebra, \emph{i.e.}\ \(\stein{\grpd} \cap \singideal \neq 0\), thus Q1 remains open. They are also non-amenable, in particular their reduced \cstar{}algebras are not nuclear (see \cite{buss-martinez-2025}*{Theorem 5.11}).\footnote{\, In fact, \(\essalg{\grpd}\) is not nuclear either for either of these examples, but this is beyond the scope of the current paper.} It would be interesting to know what happens in the nuclear case, in particular whether the result of \cite{GH2025}*{Cor. D} proving Q2 for amenable bundles of groups extends to amenable groupoids.

As a short digression from the main goal of the paper, we also show that the examples in \cref{sec:bundle-groups} are counterexamples to the Baum--Connes Conjecture (see \cref{thm:baum-connes}). In fact, the ``essential assembly map'' \(\mu_{\grpd}^{\ess} \colon K^{\rm top}_*\left(\grpd\right) \to K_*\left(\essalg{\grpd}\right)\) is not surjective. This improves on previous results of Higson, Lafforgue and Skandalis~\cite{higson-lafforgue-skandalis}*{Section 3}. 

\smallskip

\textbf{Acknowledgements:} We are grateful to Julian Gonzales for pointing us to \cite{Thom}. Moreover, the second-named author would like to thank Benjamin Steinberg and Chris Bruce for helpful conversations on topics related to the text.

\section{Preliminaries} \label[section]{sec:prelim}

\subsection{Ample groupoids} \label{subsec:grpds}
We refer the reader to \cite{SimsNotes2020} for a comprehensive survey on Hausdorff groupoids, and to \cites{ExelBigPaper,buss-martinez-2025} (and references therein) for similar introductions to non-Hausdorff ones.
As usual in the field, we view a groupoid $\grpd$ as a set of morphisms by identifying the objects with the respective identity morphisms, which we call \emph{units} and denote by $\grpd^{(0)}$. Then the \emph{source} and \emph{range} maps are respectively defined to be \(s \colon \gamma \mapsto \gamma^{-1} \gamma\) and \(r \colon \gamma \mapsto \gamma \gamma^{-1}\).
A \emph{topological} groupoid is endowed with a topology such that the multiplication, inverse, source and range maps are all continuous. In this paper we will always assume that $\grpd^{(0)}$ is Hausdorff and locally compact, however we do \emph{not} require \(\grpd\) to be globally a Hausdorff space.
A subset \(B \subseteq \grpd\) is an \editB{open} \emph{bisection} if both the source and range maps restrict to homeomorphisms from \(B\) onto their \editB{open} images in \(\grpd^{(0)}\).
The groupoid \(\grpd\) is \emph{\'etale} if its topology has a basis of open bisections, and is \emph{ample} if these bisections can also be assumed to be compact.

Throughout the paper, given \(A, B \subseteq \grpd\) we let
\begin{equation} \label[equation]{eq:convention about AB}
    A B \coloneqq \left\{ab : a \in A, b \in B \text{ and } (a, b) \in \grpd^{(2)}\right\}.
\end{equation}
Likewise, if \(A = \{a\}\) is a singleton then we define \(a B \coloneqq \{a\} B\) and, similarly, \(B a \coloneqq B \{a\}\).
In particular, \(\grpd w\) denotes all the arrows in \(\grpd\) starting at \(w \in \grpd^{(0)}\), and \(w\grpd w\) is the \emph{isotropy group} at \(w\) (necessarily discrete when $\grpd$ is étale). The set of isotropy groups forms a subgroupoid in $\grpd$ denoted by $\operatorname{Iso} \grpd$.
The subgroupoid $\grpd w \grpd$ is called the \emph{orbit} of $w$.
An étale groupoid is \emph{minimal} if all orbits are dense in the unit space, and \emph{effective} if $\operatorname{int}(\operatorname{Iso} \grpd)=\grpd^{(0)}$.

\subsection{Groupoids associated to self-similar groups} \label{subsec:selfsim_prelim}

The example in \cref{sec:minimal-effective} is a groupoid associated to a so-called self-similar group acting on a countably infinite alphabet. We provide a brief introduction to such groupoids here, see for example \cite{BenNora21}*{Section 6.1} for a more detailed one -- we also follow the notation of that paper.

Let $X$ be an infinite set (called the alphabet) and let $X^*$ and $X^\omega$ denote the finite and (right) infinite words in $X$ respectively. 
A \emph{self-similar group} on $X$ is a group $G$ together with a faithful, length-preserving action on $X^*$ such that for all $g \in G$ and $x \in X$, there exists some $h \in G$ with $g(x\gamma)=g(x)h(\gamma)$ for all $\gamma \in X^*$. In this case we denote $h$ by $g|_x$, a notation we extend to any finite word $\alpha=x_1\ldots x_n$ by $g|_\alpha:=g|_{x_1}|_{x_2}\ldots |_{x_n}$, so that we have $g(\alpha\gamma)=g(\alpha)g|_\alpha(\gamma)$ for all $\gamma \in X^*$. The action of $G$ on $X^*$ naturally extends to $X^\omega$ 
and more generally satisfies $g(\alpha w)=g(\alpha)g|_\alpha(w)$ for all $\alpha \in X^*$ and $w \in X^* \cup X^\omega$.

Recall that a semigroup \(S\) is called \emph{inverse} if for all \(s \in S\) there is a unique element \(s^* \in S\) such that \(ss^*s = s\) and \(s^*ss^* = s^*\).
We refer the reader to \cite{Lawson} for a comprehensive introduction to inverse semigroups, or to \cite{LawsonPrimer} for a shorter one. An inverse \emph{monoid} is an inverse semigroup with an identity element $1 \in S$.

There is an inverse monoid $S$ associated to the self-similar group $G$ acting on $X^*$ (see \cite{BenNora21}*{Section 6.1}). This \(S\) consists of elements of the form
\[
S \coloneqq \left\{\alpha g \beta^\ast \colon \alpha, \beta \in X^\ast, g \in G\right\} \cup \{0\},
\]
where $(x_1\ldots x_n)^*=x_n^* \ldots x_1^*$ should be thought of as a formal inverse of $x_1 \ldots x_n \in X^*$.\footnote{\, Let us hold a moment of silence for the unfortunate convention of $*$ being used both to denote the free monoid as well as adjoint/inverses, preventing us from introducing a notation for formal inverses of words in $X^{*}$.}
We view $G$ and $X^*$ as subsets of $S$ via the embeddings $G \ni g \mapsto \epsilon g \epsilon^*$ and $X \ni x \mapsto x1_G\epsilon^*$ where $\epsilon$ denotes the empty word. In particular $\epsilon1_G\epsilon^*=\epsilon=\epsilon^*=1_G$ is the identity element of $S$, denoted by $1$.
The multiplication on $S$ extends the multiplication on $G$ and juxtaposition on $X^*$, and
for any $x,y \in X$ and $g \in G$, we define
$$x^*y \coloneqq \delta_{x,y},\ gx \coloneqq g(x)g|_x\ \hbox{ and } x^*g \coloneqq g|_{g^{-1}(x)} (g^{-1}(x))^*.$$
It follows that for any $\alpha \in X^*$, we have the equality $g\alpha=g(\alpha)g|_\alpha$ in $S$, which we will frequently use. The inverse of $\alpha g \beta^*$ is $\beta g^{-1} \alpha^*$.

The action of $G$ on $X^* \cup X^\omega$ naturally extends to an action of $S$ by the partial homeomorphisms
\begin{align*}
  \alpha g \beta^* \colon D(\beta) & \to D(\alpha), \\
  \beta w &\mapsto \alpha g(w),
\end{align*}
where $D(\gamma) \coloneqq \gamma X^* \cup \gamma X^\omega$.
This is in fact none other than the spectral action of $S$ on the character space of its idempotent semilattice.
\editB{Because the action of \(G\) on \(X^*\) is assumed faithful, t}he groupoid $\grpd$ associated to the self-similar group $G$ is the groupoid of germs of this action (and thus the universal groupoid of $S$ -- see \cite{PatersonGroupoids}), which we also describe explicitly:
$$\grpd=\{(\alpha g \beta^*,w)\colon w \in D(\beta)\}/\sim$$
where \((t_1, w) \sim (t_2, w)\) if there is some $\gamma \in X^*$ such that \(w \in D(\gamma)\) and \(t_1 \gamma = t_2 \gamma\).
The class of $(t,w)$ is called a \emph{germ} and is denoted by $[t,w]$. The unit space of \(\grpd\) is $\{[1,w] \colon w \in X^* \cup X^\omega\}$, which is naturally identified with $X^* \cup X^\omega$. The source and range maps are then given as $s([t,w])=w$, $r([t,w])=t(w)$, and the groupoid operations are defined as \([t_2,t_1(w)] [t_1, w] = [t_2t_1,w]\) and \([t_1,w]^{-1} = [t_1^*, t_1(w)]\) for all \([t_2,t_1(w)], [t_1,w] \in \grpd\). 

A basis of compact open  sets in $\grpd^{(0)}$ consists of sets of the form $D(\alpha) \cap D(\alpha\beta_1)^c \cap \dots \cap D(\alpha\beta_n)^c$  where $\alpha, \beta_1, \ldots, \beta_n \in X^*$. Setting $(\alpha g \beta^*, U)=\{[\alpha g \beta^*, w]\colon w \in U\}$, a basis of compact open bisections of $\grpd$ is 
$$\{(\alpha g \beta^*, U) \colon \alpha g \beta^* \in S,\ U \subseteq D(\beta) \hbox{ is a basic compact open set in } X\}.$$
In particular $\grpd$ is an ample groupoid.
Given $s=\alpha g \beta^*$, we sometimes denote $D(\beta)$ by $D(s^*s)$ (which coincides with the usual meaning of $D(s^*s)$ in universal groupoids), in particular $(s, D(s^*s))=(\alpha g \beta^*, D(\beta))$ is the canonical compact open bisection associated to $s$.

\subsection{Algebras associated to groupoids} \label{subsec:algs}
Let \(\grpd\) be a fixed ample groupoid with locally compact Hausdorff unit space \(\grpd^{(0)}\).
Recall that given a function $\ff \colon \grpd \to \CC$, 
\[
  \supp{\ff} \coloneqq \left\{\gamma \in \grpd : \ff\left(\gamma\right) \neq 0\right\}
\]
denotes its (\editB{strict\footnote{\, Some sources -- keeping the convention from continuous functions -- refer to this as the open support.}}) \emph{support}.
Let $B_c(\grpd)$ denote the set of bounded Borel measurable functions $\grpd \to \CC$ whose support is contained in some compact set.\footnote{\, This is \emph{different} from saying that the closed support is compact, as \(\grpd\) may fail to have an abundance of closed compact sets.}
The \emph{Steinberg algebra of \(\grpd\)}, denoted by \(\stein \grpd\), is
\[
  \stein \grpd \coloneqq \spn \left\{\chi_U : U \subseteq \grpd \text{ compact open bisection}\right\} \subseteq B_c(\grpd).
\]
The \emph{algebra associated to \(\grpd\)}, denoted by \(\algalg \grpd\), is
\[
  \algalg \grpd \coloneqq \spn \left\{\ff \in \contc{U} : U \subseteq G \text{ open bisection}\right\} \subseteq B_c(\grpd),
\]
where \(\contc{U}\) denotes the set of functions \(U \to \CC\) that are continuous and compactly supported.
One should note that in both formulas above we tacitly extend functions \(\ff \colon U \to \CC\) supported on an open bisection \(U \subseteq \grpd\) to all of \(\grpd\) by extending by \(0\) outside of \(U\).
Observe that \(\stein \grpd \subseteq \algalg \grpd\) by definition.
Both these sets form complex algebras with respect to pointwise addition and a \emph{convolution product} defined by
$$(f \ast g)(\gamma) \coloneqq \sum_{\alpha\beta=\gamma} f(\alpha)g(\beta).$$ 
In particular, observe that $\chi_U \ast \chi_V=\chi_{UV}$ for any compact open bisections $U, V$, and $f \ast \chi_U=f|_{\grpd U}$ for any compact open set $U \subseteq \grpd^{(0)}$.

The following results provide a useful description of the algebras \(\algalg \grpd\) and \(\stein \grpd\).
\begin{proposition}[see \cite{BenSteinAlgPaper}*{Theorem 5.3}] \label[proposition]{prop:stein description}
  If \(S\) is an inverse semigroup with zero and \(\grpd\) its universal groupoid, then \(\stein\grpd = \spn \{\chi_{(s,D(s^*s))} \colon s \in S\setminus \{0\}\}\).
\end{proposition}

\begin{proposition}[see \cite{ExelBigPaper}*{Proposition 3.10}] \label[proposition]{prop:alg alg description}
  If \(\grpd\) is an ample groupoid and \(\CU\) is a cover of \(\grpd\) consisting of compact open bisections, then \(\algalg\grpd = \spn \{\contc{U} \colon U \in \CU\}\).
\end{proposition}

One can see \editB{elements of} both \(\stein \grpd\) and \(\algalg \grpd\) as operators via the canonical ``left regular representations'' associated to a unit \(w \in \grpd^{(0)}\), \emph{i.e.}\ the representation defined by
\begin{align*}
  \lambda_w \colon \algalg{\grpd} & \to \CB\left(\ell^2\left(\grpd w\right)\right), \\
  \lambda_w\left(\ff\right) (\delta_\gamma) & \coloneqq \sum_{\alpha \in \grpd \gamma \gamma^{-1}} \ff\left(\alpha\right) \delta_{\alpha\gamma}
\end{align*}
for all \(\delta_\gamma \in \ell^2(\grpd w)\). (Recall that, under our convention, \(\grpd w\) denotes all the elements \(\gamma \in \grpd\) whose source is \(w\).)
It is routine to show that \(\lambda_w\) is a representation for all \(w \in \grpd^{(0)}\). In fact, if \(\grpd\) is a group then \(\lambda_w\) is the usual left regular representation. Moreover, 
\[
  \rednorm{\ff} \coloneqq \sup_{w \in \grpd^{(0)}} \norm{\lambda_w\left(\ff\right)}
\]
defines a \cstar{}norm on \(\algalg \grpd\) (and on \(\stein \grpd\) as well, naturally).
As it turns out, the completions of \(\algalg \grpd\) and \(\stein \grpd\) with respect to \(\rednorm{\cdot}\) are the same, \emph{i.e.}\ the \emph{reduced \cstar{}algebra of \(\grpd\)}, denoted by \(\redalg \grpd\).
In order to lighten the notation, we shall henceforth denote \(\rednorm{\cdot}\) simply by \(\norm{\cdot}\).
\begin{proposition}\label[proposition]{prop:j-map}
  Let \(\grpd\) be an \'etale groupoid with locally compact Hausdorff unit space. The inclusion \(j_{\alg} \colon \algalg \grpd \hookrightarrow \borelb{\grpd}\) extends uniquely to a contractive, injective map
  \[
    j \colon \redalg{\grpd} \hookrightarrow \borelb{\grpd},
  \]
  where \(\borelb{\grpd}\) is the \cstar{}algebra of Borel bounded functions \(\grpd \to \CC\) equipped with supremum norm, pointwise addition and product.
  Furthermore, the map \(j\) satisfies that \(j(\fa\fb) = j(\fa)*j(\fb)\) and \(j(\fa^*) = j(\fa)^*\) for all \(\fa, \fb \in \redalg{\grpd}\).
\end{proposition}
\begin{proof}
    The first statement is given in \cite{buss-martinez-2025}*{Lemma 3.27} (whose proof is identical to that of \cite{Renault80}*{II.4.2}). The ``furthermore'' statement follows since it is satisfied by \(j_{\alg}\) itself \editB{and some routine $\varepsilon/2$-arguments.}
\end{proof}

For all intents and purposes, by \cref{prop:j-map} we may consider elements \(\ff \in \redalg{\grpd}\) as functions \(\ff \colon \grpd \to \CC\). 
We shall henceforth tacitly do this.
The following, which will be useful for us later on, is a natural consequence of this point of view.
\begin{proposition} \label[proposition]{cauchy-plus-convergent}
  Suppose that \(\{\ff_n\}_{n \in \NN} \subseteq \algalg \grpd\) and \(\ff \in \borelb{\grpd}\) are such that
  \begin{enumerate}
    \item \label{cauchy-plus-convergent:1} \(\supnorm{\ff_n - \ff} \to 0\) when \(n \to \infty\), where \(\supnorm{\cdot}\) denotes the supremum norm; and
    \item \label{cauchy-plus-convergent:2} \(\{\ff_n\}_{n \in \NN}\) is Cauchy in the reduced norm \(\norm{\cdot}\).
  \end{enumerate}
  Then \(\ff \in \redalg{\grpd}\) and, moreover, \(\ff_n \to \ff\) in \(\norm{\cdot}\).
\end{proposition}
\begin{proof}
  From (\ref{cauchy-plus-convergent:2}) we have that \(\ff_n \to \fg\) for some element \(\fg \in \redalg{\grpd}\). By \cref{prop:j-map} it follows that \(\supnorm{\cdot} \leq \norm{\cdot}\) and, thus,
  \[
    \supnorm{\fg - \ff} \leq \supnorm{\fg - \ff_n} + \supnorm{\ff_n - \ff} \leq \norm{\fg - \ff_n} + \supnorm{\ff_n - \ff} \to 0.
  \]
  As \(\supnorm{\cdot}\) separates the points of \(\borelb{\grpd}\), and hence of \(\redalg{\grpd}\) as well, it follows that \(\ff = \fg \in \redalg{\grpd}\).
\end{proof}

We end the subsection discussing singular functions. We call a function \(\ff \in \redalg \grpd\) \emph{singular} if $\supp \ff$ has empty interior.
It was shown in \cite{CEPSS-2019} that the set \(\singideal \subseteq \redalg{\grpd}\) of singular functions forms an ideal of \(\redalg \grpd\), and the quotient \(\essalg \grpd \coloneqq \redalg \grpd /\singideal\) is usually called the \emph{essential \cstar{}algebra of \(\grpd\)} (see \emph{e.g.}\ \cites{buss-martinez-2025,hausdorff-covers-2025,bruce-li:2024:alg-actions,ExelPitts,KwasniewskiMeyer-essential-cross-2021}).
We will call the elements \(\ff \in \asingideal \coloneqq \algalg \grpd \cap \singideal\) \emph{algebraic} singular functions.


\subsection{Spectral radii in non-amenable groups} \label{subsec:radii}
We introduce a technical statement about group \cstar{}algebras which will be crucial for us later.
Let \(H\) be a finitely generated non-amenable group. For a finite symmetric subset \(K \subseteq H\), we denote the \emph{spectral radius of the random walk associated with $K$} by
\[
  \rho\left(K\right) \coloneqq \norm{\frac{1}{K}\sum_{h \in K} h},
\]
where the above norm is the reduced norm, \emph{i.e.}\ the operator norm in \(\CB(\ell^2(H))\).
It is shown in \cite{Thom}*{Theorem 1} that for any finitely generated non-amenable group $H$ and any $\varepsilon > 0$, there exist a symmetric subset $K_\varepsilon \subseteq H$ such that $\rho(K_\varepsilon) < \varepsilon.$
In particular, this implies the following.
\begin{proposition} \label[proposition]{prop:nonamenable-groups-can-scatter}
If $H$ is any countably infinite non-amenable group, then there exists a sequence $\{H_n\}_{n \in \NN}$ of finite subsets of $H$ such that $\frac{1}{\abs{H_n}}\sum_{h \in H_n} h$ converges to $0$ in $\redalg{H}$.
\end{proposition}

\section{The strategy} \label{sec:strategy}
In both \cref{sec:bundle-groups,sec:minimal-effective} we follow the same basic strategy to show that the set of algebraic singular functions is not dense in the \cstar{}ideal of singular functions. We outline the main steps below. Throughout the section, $\grpd$ denotes an ample groupoid, which plays the role of the respective counterexamples constructed in \cref{sec:bundle-groups,sec:minimal-effective}.

\textbf{Step 1:} \textbf{scattering.}
Let $B \subseteq \grpd^{(0)}$ be a set. Suppose that there exists a sequence $\{\CB_n\}_{n \in \NN}$ of finite sets of compact open bisections in $\grpd$ such that for any $n \in \NN$ and any pair of distinct bisections $B_1, B_2 \in \CB_n$ we have $B_1 \cap B_2 =B$. We may then consider 
\[
  \fb_n \coloneqq \frac{1}{\abs{\CB_n}}\sum_{U \in \CB_n}\chi_{U} \in \steinalg\grpd \subseteq \algalg{\grpd}.
\]
Now let $\gamma \in \grpd$. If $\gamma \in B$, then $\gamma \in \bigcap \CB_n$ and so $\fb_n(\gamma) = 1$. If $\gamma \notin B$, then there is at most one set $B_i \in \CB_n$ with $\gamma \in B_i$ and so $\abs{\fb_n(\gamma)} \leq 1/|\CB_n|$. Therefore, if $\abs{\CB_n} \to \infty$ as $n \to \infty$, then $\fb_n \rightarrow \chi_B$ in supremum norm.
In particular, if $\{\fb_n\}_{n \in \NN}$ is, furthermore, Cauchy in $\redalg\grpd$, then $\fb_n \rightarrow \chi_B$ in operator norm, and hence $\chi_B \in \redalg\grpd$ (see \cref{cauchy-plus-convergent}).
Observe that while $B$ is open (as it is an intersection of finitely many open sets), it does not have to be compact (as in non-Hausdorff spaces the intersection of compact spaces need not be compact). Thus $\chi_B$ is typically not an element of $\steinalg\grpd$ or even $\algalg\grpd$.
We refer to this technique as \emph{scattering}, and in both examples in \cref{sec:bundle-groups,sec:minimal-effective} it enables us to construct such elements $\fb = \chi_B$ in $\redalg{\grpd}$ where $B \subseteq \grpd^{(0)}$ is a non-compact open set.

\textbf{Step 2:} \textbf{finding a ``$B$-singular'' function}. Given an open subset $U \subseteq \grpd^{(0)}$, we say that $\fa \in \steinalg\grpd$ is \emph{$U$-singular} if $\fa * \chi_U$ is singular. Observe that if $U = B$ is as in Step 1, then in particular $\fa * \chi_B = \fa * \fb=\lim_{n \to \infty} \fa *\fb_n$ is singular. However, if $\fa * \fb_n \in \steinalg\grpd$ is non-singular, then there is no obvious reason for $\fa*\fb$ to be in the closure of $\asingideal = \singideal \cap \algalg{\grpd}$.
In both our examples, we define $B$-singular functions $\fa \in \steinalg\grpd$ which have a large support outside $B$, thus ensuring the non-singularity of $\fa * \fb_n$ for all \(n \in \NN\).

\textbf{Step 3:} \textbf{showing $\fa*\fb$ is not in $\ol\asingideal$}. To formally prove that $\fa*\fb$ is not in the closure of $\asingideal$, we show something stronger: that $\fa*\fb$ is not even in the supremum closure of $\asingideal$. In both cases, our construction ensures that if $\ff \in \steinalg\grpd$ is  such that $\supp {\fa*\fb} \subseteq \supp \ff,$ then $\supp \ff$ necessarily contains an open set, which immediately implies that  $\fa*\fb$ cannot be approximated by singular functions of $\steinalg\grpd$ in supremum norm (recall that the values taken by $\fa*\fb$ are by construction discrete in $\CC$). The analogous statement for $\algalg \grpd$ follows from \editB{the fact that $J \cap   \steinalg \grpd$ is dense in $\asingideal$ \cite[Thm. 6.3]{GH2025}}. In \cref{sec:minimal-effective}, the proof is explicitly broken down into these two steps; in \cref{sec:bundle-groups}, we give a direct proof for $\algalg \grpd$.

\begin{remark}
We note the following in passing.
\begin{enumerate}
  \item Examples built following the strategy above always have algebraic singular functions, even in the Steinberg algebra.
  Indeed, if $\gamma \in \supp{\fa*\fb}$, then there exists some compact open neighborhood $U$ of $s(\gamma) = \gamma^{-1}\gamma$ contained in $B$. Then $\fa*\chi_U \in \stein\grpd$ is singular, since it is a restriction of $\fa*\fb$, but non-zero, as $\gamma \in \supp{\fa*\chi_U}$.
  \item In Step 1, we can only ensure that $\{\fb_n\}_{n \in \NN}$ is Cauchy at the cost of the amenability of $\grpd$ via \cref{prop:nonamenable-groups-can-scatter}. We suspect that non-amenability is necessary for our strategy to work. In \emph{Hausdorff} amenable groupoids, it is well-known that any \(\ff \in \redalg{\grpd}\) can be approximated by \(\algalg{\grpd} \ni \ff_n \to \ff\) such that \(\supp{\ff_n} \subseteq \supp{\ff}\) (cf.\ \cite{buss-martinez-2025}*{Proposition 5.15}).
  In our examples, $\fb = \chi_B$ dramatically fails this property.
\end{enumerate}
\end{remark}

\section{First example: a bundle of groups} \label{sec:bundle-groups}
The first example we construct is a groupoid \(\grpd\) that is a bundle of groups (that is, \(s(\gamma) = \gamma^{-1} \gamma = \gamma\gamma^{-1} = r(\gamma)\) for all \(g \in \grpd\)). In particular, it is neither effective nor minimal.

\subsection{Description of the groupoid} \label{subsec:description-grpd}
Let \(H\) be a discrete non-amenable group, and let \(\ZZ_2 = \{0,1\}\) be the group with two elements. Put \(G \coloneqq \ZZ_2 \times H\).
Let \(X \coloneqq \{x_{ij}\}_{i, j \in \NN}, Y \coloneqq \{y_i\}_{i \in \NN}\) and \(Z \coloneqq \{z_i\}_{i \in \NN}\) be three countably infinite sets, and let \(\epsilon\) be an additional symbol.
We define our groupoid $\grpd$ as a bundle of groups on the unit space  
\(\grpd^{\left(0\right)} = X \sqcup Y \sqcup Z \sqcup \{\epsilon\},\)
where the isotropy groups at $\epsilon$ and $Z$ are isomorphic to $G$, the isotropies at $Y$ are isomorphic to $\ZZ_2$, and $X$ has trivial isotropies. Formally, the groupoid \(\grpd\) as a set is
\[
  \grpd = X \sqcup (\ZZ_2 \times Y) \sqcup (G \times Z) \sqcup (G \times \left\{\epsilon\right\}) ,
\]
and we identify \(\epsilon\), \(z \in Z\)  and $y \in Y$ with \((0, 1_H, \epsilon)\), \((0, 1_H, z)\) and $(0,y)$ respectively.
Given any pairs $(g,u), (h,u)$ in either $\ZZ_2 \times Y$, $G \times Z$, or $G \times \left\{\epsilon\right\}$, we define $s(g,u)=r(g,u)=u$ and $(g,u)(h,u)=(gh,u)$, that is, the operation within each fiber is inherited from $H$ and $\ZZ_2$.

The topology on $\grpd^{(0)}$ is the following: the points $x \in X$ and $z \in Z$ are discrete; 
$\{x_{ij}\}_{j=1}^\infty$ converges to $y_i$\editB{; $\{x_{ij}\}_{i = 1}^\infty$ converges to $\epsilon$ for all $j$;} and $\{y_i\}_{i=1}^\infty$ and  $\{z_i\}_{i=1}^\infty$ both converge to $\epsilon$ \editB{too}. 
Putting $X_{i} \coloneqq \{x_{ij}\colon j \in \NN\}$, a basis of compact open sets of $\grpd^{(0)}$ is given by sets of the form
\begin{enumerate}
\item $\{x\}$ for any $x \in X$ and $\{z\}$ for any $z \in Z$;
\item for any $y_i$, $i \in \mathbb N$ and any $F \Subset X_i$ finite set, $U_{y_i,F}\coloneqq\{y_i\} \cup X_i \setminus F$;
\item for any $F_Y \Subset Y$, $F_Z \Subset Z$ finite sets, 
$$U_{F_Y,F_Z}\coloneqq \grpd^{(0)} \setminus  \left(\bigcup_{y \in  F_Y} U_{y,\emptyset} \cup  F_Z\right).$$
\end{enumerate}
It is routine to verify that these sets are compact and form a basis with respect to the topology they generate, which coincides with the topology described above. Now for $g=(n,h) \in G$, let 
$$U_g \coloneqq X \sqcup (n \times Y) \sqcup (g \times Z) \sqcup (g \times \left\{\epsilon\right\}),$$
then a basis of compact open sets of $\grpd$ consists bisections of the form $U_gU$ where $U$ is any basic compact open set in $\grpd^{(0)}$. It follows that $\grpd$ is étale, moreover ample as it has a basis of compact open bisections. It is also clear that $\grpd$ is countable and second countable.
Observe that \(x_{ij}\) converges in \(j\) to both \((0, y_i)\) and \((1,y_i)\), and converges in \(i,j\) to every \((g, \epsilon)\). In particular, \(\grpd\) is not Hausdorff, in fact the set of dangerous units is exactly \(Y \cup \{\epsilon\}\).

The rest of the section is dedicated to proving the following.
\begin{theorem} \label[theorem]{thm:alg-sing-are-dense}
  If \(H\) is a finitely generated non-amenable group, and \(\grpd\) is as above, then \(\asingideal = \singideal \cap \algalg{\grpd}\) is not dense in \(\singideal\).
  In particular, neither is \(\singideal \cap \stein \grpd\).
\end{theorem}

\subsection{Scattering}
Since \(H\) is non-amenable, by \cref{prop:nonamenable-groups-can-scatter} we can find some sequence \(\{H_n\}_{n \in \NN}\) of finite non-empty subsets of \(H\) such that the sequence \(\frac{1}{\abs{H_n}} \sum_{h \in H_n} h\) tends to \(0\) in \(\redalg{H}\).
Viewing \(H \subseteq G\), we have that \(\redalg{H} \subseteq \redalg{G}\), and thus \(\frac{1}{\abs{H_n}} \sum_{h \in H_n} h\) tends to \(0\) in \(\redalg{G}\) as well.

Let
\[
\fb_n \coloneqq \iota\left(\frac{1}{\abs{H_n}} \sum_{h \in H_n} h\right) = \frac{1}{\abs{H_n}} \sum_{h \in H_n} \chi_{U_{(0,h)}} \in \stein \grpd \subseteq \algalg{\grpd}.
\]

To show that $\fb_n$ converges in norm, we use the following lemma, which is well-known and routine to check.

\begin{lemma}
\label{lem:bundle-convergence}
Let $\grpd$ be a bundle of groups, and let $\varphi \colon \grpd w = w \grpd w \to K$ be a group isomorphism. There is an induced (spatially implemented) isomorphism $\hat \varphi \colon \CB(\ell^2(\grpd w)) \to \CB(\ell^2(K))$ such that
$$\lambda_w(\chi_{U})=\hat\varphi^{-1}(\lambda_K(\varphi(U \cap w \grpd w)))$$
for every $U \subseteq \grpd$ compact open bisection, where $\lambda_w \colon \algalg\grpd \to \CB(\ell^2(\grpd w))$
and $\lambda_K \colon K \to \CB(\ell^2(K))$ are the regular representations of \(\grpd\) (at \(w\)) and of \(K\) respectively. That is, \(\hat\varphi\) intertwines \(\lambda_w\) and \(\lambda_K\).
\end{lemma}

\begin{proposition} \label[proposition]{cauchy-bundle-groups}
  The sequence \(\{\fb_n\}_{n \in \NN}\) is Cauchy in \(\redalg{\grpd}\). In fact, it converges to \(\chi_B\), where \(B \coloneqq X \sqcup Y\).
\end{proposition}
\begin{proof}
In order to show that \(\{\fb_n\}_{n \in \NN}\) is Cauchy, it suffices to show that \(\{\lambda_x(\fb_n)\}_{n \in \NN}, \{\lambda_y(\fb_n)\}_{n \in \NN}\) and \(\{\lambda_z(\fb_n)\}_{n \in \NN}\) are Cauchy for any (equivalently all) \(x \in X, y \in Y\) and \(z \in Z\).
(In particular, notice that the behaviour at \(\epsilon\) is identical to that at any \(z \in Z\).)
Let $\varphi_x,\varphi_y,\varphi_z$ denote the isomorphisms $x\grpd x \to \{1\}, y\grpd y \to \mathbb Z_2, z\grpd z \to G$ respectively; and as in \cref{lem:bundle-convergence}, consider the induced isomorphisms $\hat\varphi_x, \hat\varphi_y, \hat\varphi_z$. It follows that
\begin{eqnarray}
\label{eqna:bundle_itemx} \lambda_x(\chi_{U_{(0,h)}})&=&\hat\varphi_x^{-1} \circ\lambda_{\{1\}}(1)=\id_{\ell^2(\grpd x)} \in \CB(\ell^2(\grpd x)); \\
\label{eqna:bundle_itemy} \lambda_y(\chi_{U_{(0,h)}})&=&\hat\varphi_y^{-1}\circ\lambda_{\ZZ_2}(0)=\id_{\ell^2(\grpd y)} \in \CB(\ell^2(\grpd y)); \\
\label{eqna:bundle_itemz} \lambda_z(\chi_{U_{(0,h)}})&=&\hat\varphi_z^{-1}\circ\lambda_{G}(0,h).
\end{eqnarray}
Equations \eqref{eqna:bundle_itemx} and \eqref{eqna:bundle_itemy} imply that 
$$\lambda_x(\fb_n)= \frac{1}{\abs{H_n}} \sum_{h \in H_n}\id_{\ell^2(\grpd x)} \quad \text{ and } \quad \lambda_y(\fb_n)= \frac{1}{\abs{H_n}} \sum_{h \in H_n}\id_{\ell^2(\grpd y)}$$
are constant, and therefore Cauchy.
Furthermore from \eqref{eqna:bundle_itemz} we obtain
$$\lambda_z\left(\fb_n\right) = \frac{1}{\abs{H_n}} \sum_{h \in H_n} \hat\varphi_z^{-1}\circ\lambda_{G}(0,h)=\hat\varphi_z^{-1} \circ \lambda_{G} \left( \frac{1}{\abs{H_n}}\sum_{h \in H_n} (0,h) \right).$$
Observe that the norm of $\lambda_{G} \left( \frac{1}{\abs{H_n}}\sum_{h \in H_n} (0,h) \right)$ in $\CB(\ell^2(G))$ is exactly the norm of $\frac{1}{\abs{H_n}}\sum_{h \in H_n} (0,h)$ in $\redalg G$ by definition, which, by assumption, is Cauchy and converges to \(0\) as $n \to \infty$.
It follows that \(\{\fb_n\}_{n \in \NN} \subseteq \redalg{\grpd}\) converges to some function \(\fb \in \redalg{\grpd}\).

By \cref{cauchy-plus-convergent}, in order to finish the proof it thus suffices to show that \(\supnorm{\fb_n - \chi_B} \to 0\) when \(n \to \infty\).
For this, observe that
\begin{itemize}
  \item \(\fb_n(x) = 1\) for all \(x \in X\);
  \item \(\fb_n(0, y) = 1\) for all $y \in Y$;
  \item \(\fb_n(1, y) = 0\) for all $y \in Y$; and
  \item \(\fb_n(g, z) = \fb_n(g, \epsilon) = 1/\abs{H_n}\) for all $g \in G$ and $z \in Z$.
\end{itemize}
In particular it follows that \(\fb_n\) restricts to \(1\) on \(B = X \sqcup Y\), and tends to \(0\) on the rest, as desired.
\end{proof}

\subsection{Finding a \texorpdfstring{\(B\)}{B}-singular function.}
Consider the element
\[
  \fa \coloneqq \chi_{U_{(0,1_H)}} - \chi_{U_{(1, 1_H)}} \in \stein{\grpd},
\]
and also consider the convolution \(\fa * \fb = \fa|_{\grpd B} \in \redalg{\grpd}\).
Then $\supp{\fa*\fb} \subseteq \grpd B \cap (U_{(0,1_H)} \cup U_{(1,1_H)})$, and so we have
\begin{itemize}
  \item \((\fa*\fb)(x) = 0\) for all \(x \in X\);
  \item \((\fa*\fb)(0, y) = 1\) for all $y \in Y$;
  \item \((\fa*\fb)(1, y) = -1\) for all $y \in Y$; and
  \item \((\fa*\fb)(g, z) = (\fa * \fb)(g, \epsilon) = 0\) for all $g \in G$ and $z \in Z$.
\end{itemize}
In particular, it follows that \(\supp{\fa*\fb}\) contains no open set, and so $\fa*\fb \in J$.

\subsection{Showing \texorpdfstring{$\fa *\fb$}{a*b} is not in \texorpdfstring{$\overline{J_{alg}}$}{Jalg}.}
The last step to showing \cref{thm:alg-sing-are-dense} is to prove that \(\fa*\fb\) is not contained in the closure of \(\singideal \cap \algalg{\grpd}\) under the supremum norm.

\begin{lemma} \label[lemma]{lem:bundle_being close to singular makes you unsingular}
  Let \(\ff \in \algalg{\grpd}\) be such that \(\abs{\ff(y)} \geq 1/2\) for all \(y \in Y\). Then \(\ff\) is not singular.
\end{lemma}

\begin{proof}
Observe that $\{U_g \colon g \in G\}$ is a compact open cover of $\grpd$, and thus by \cref{prop:alg alg description} we may write any $\ff \in  \algalg{\grpd}$ as
$$\ff=\sum_{g \in G} \ff_g$$
where $\ff_g \in C_c(U_g)$, and only finitely many -- say $C$ -- summands are nonzero.

Let $\varepsilon \coloneqq \frac{1}{5C} < \frac{1}{4C}$, and consider $(g, \epsilon) \in U_g$. Since $\ff_g$ is continuous on $U_g$,  $(g, \epsilon)$ has some basic open neighborhood $U_{g, \varepsilon}$ such that for any $\gamma \in U_{g, \varepsilon}$ we have
\begin{equation}
\label{eq:bundle-epsilon-close}
|\ff_g(g, \epsilon)-\ff_g(\gamma)| < \varepsilon.
\end{equation}
We may assume that $U_{g, \varepsilon}$ is a basic compact open bisection that is of the form $U_gU_{F^g_Y, F_Z^g}$ for some finite sets $F^g_Y \Subset Y$, $F^g_Z \Subset Z$.
Put $F_Y \coloneqq \bigcup_{g \in G, \ff_g \neq 0} F_Y^g$ and $F_Z \coloneqq \bigcup_{g \in G, \ff_g \neq 0} F_Z^g$, 
and let $y \in Y \setminus F_Y$, $z \in Z \setminus F_Z$.
Observe that we have
\begin{equation*}
\frac{1}{2}\leq|\ff(y)|=\left|\sum_{h \in H} \ff_{(0,h)}(y)\right| \leq \sum_{h \in H} |\ff_{(0,h)}(y)|,
\end{equation*}
where the first inequality is by assumption, and the equality follows from $\supp{\ff_g} \subseteq U_g$ and the observation that $y\in U_g$ if and only if $g=(0,h)$ for some $h \in H$.
Since there are at most $C$ nonzero terms on the right hand side, there exists some $h \in H$ such that $|\ff_{(0,h)}(y)|\geq \frac{1}{2C}$. Put $g \coloneqq (0,h)$, then $y\notin F_Y \supseteq F_Y^g$ and $z \notin F_Z$ ensure that $y, (g,z) \in U_{g, \varepsilon}$, hence
we obtain
$$\frac{1}{2C}-|\ff_g(g,z)|\leq|\ff_g(y)|-|\ff_g(g,z)|\leq
|\ff_g(y)-\ff_g(g,z)|\leq |\ff_g(y)-\ff_g(g,\epsilon)|+|\ff_g(g,\epsilon)-\ff_g(g,z)| <2\varepsilon,
$$
where the last inequality is by \eqref{eq:bundle-epsilon-close}.
In particular, as $\frac{1}{2C}-2\varepsilon >0$, we have
$\ff_g(g,z)\neq 0$. But notice that $\supp{\ff_g} \subseteq U_g$ implies $\ff(g,z)=\ff_g(g,z)$, and so $(g,z) \in \supp \ff$, which is an open set and therefore $\ff$ is nonsingular.
\end{proof}

Using \cref{lem:bundle_being close to singular makes you unsingular}, the proof of \cref{thm:alg-sing-are-dense} is now easy. Indeed, suppose that \(\fa*\fb \in \redalg{\grpd}\) is the supremum norm limit of some \(\{\ff_n\}_{n \in \NN} \subseteq \algalg{\grpd}\).
Since \(\abs{\fa*\fb(y)} = 1\) for all \(y \in Y\), it follows that \(\abs{\ff_n(y)} \geq 1/2\) for all large enough \(n \in \NN\).
\cref{lem:bundle_being close to singular makes you unsingular} then implies that \(\ff_n\) is not singular, as desired.

%
%

\subsection{The groupoid \texorpdfstring{\(\grpd\)}{G} is a counterexample to the (essential) Baum--Connes Conjecture}
In this subsection we will prove that the groupoid \(\grpd\) is another counterexample to the Baum--Connes Conjecture. In fact, it is a counterexample to what we term as the \emph{essential} Baum--Connes Conjecture.
We refer the reader to \cite{higson-lafforgue-skandalis} for a short introduction to the Baum--Connes Conjecture and some of the previous counterexamples, and to \cite{tu-1999} for a treatise of this conjecture in the case when \(\grpd\) is Hausdorff.
In particular, it involves a \emph{reduced assembly map}
\begin{equation} \label{eq:red-assembly-map}
  \mu_\grpd^{\red} \colon K^{\rm top}_*\left(\grpd\right) \to K_*\left(\redalg{\grpd}\right),
\end{equation}
and asks whether it is an isomorphism of abelian groups. 
The authors of \cite{higson-lafforgue-skandalis} use several compatibility relations of \(\mu^{\red}_\grpd\) to deduce \cite{higson-lafforgue-skandalis}*{Proposition 2}, yielding a plethora of counterexamples to \eqref{eq:red-assembly-map} being \emph{surjective}.
These techniques are the main ideas we will use to prove that \(\grpd\) violates the Baum--Connes Conjecture.
In fact, see \cref{thm:baum-connes}, these ideas in \cite{higson-lafforgue-skandalis} will let us prove even more, that the composition of \(\mu_\grpd^{\red}\) with the map induced on \(K_0\) by the canonical quotient map \(\Lambda \colon \redalg{\grpd} \to \essalg \grpd\) is also not surjective.
Note that the counterexamples in \cite{higson-lafforgue-skandalis} are either with no nonzero singular functions (and thus \(\redalg{\grpd} = \essalg{\grpd}\)), or such that the composition
\begin{equation} \label{eq:ess-assembly-map}
  \mu_\grpd^{\ess} \colon K^{\rm top}_*\left(\grpd\right) \xrightarrow{\mu_\grpd^{\red}} K_*\left(\redalg \grpd\right) \xrightarrow{K_*\left(\Lambda\right)} K_*\left(\essalg{\grpd}\right)
\end{equation}
is surjective (see \cite{higson-lafforgue-skandalis}*{Section 2, 3rd Counterexample}). 

For the reader's convenience, and (not quite, but almost) following the notation of \cite{higson-lafforgue-skandalis}, we let:
\[
  F \coloneqq \left\{\epsilon\right\} \sqcup \left\{z_i : i \in \NN\right\} \quad \text{ and } \quad V \coloneqq \left\{x_{ij}, y_i : i,j \in \NN\right\}.
\]
Note that \(F\) is closed, \(V\) is open, and they partition \(\grpd^{(0)}\).
Likewise, let \(\grpd_F \coloneqq F \grpd F\) and \(\grpd_V \coloneqq V \grpd V\) be the restriction of \(\grpd\) to these groupoids.
It follows from arguments similar to those in \cite{BenAli}*{Corollary 4.4} or simply by routine checking, that the restriction map \(\rest_F \colon \stein{\grpd} \to \stein{\grpd_F}\) yields a short exact sequence
\[
  0 \to \stein{\grpd_V} \to \stein{\grpd} \to \stein{\grpd_F} \to 0.
\]
Using the same arguments as in \cite{higson-lafforgue-skandalis}*{pp.\ 334} we see that the above short exact \editB{sequence} induces a short exact sequence of the maximal \cstar{}algebras:
\begin{equation} \label{eq:max-short-exact-seq}
  0 \to \maxalg{\grpd_V} \to \maxalg{\grpd} \to \maxalg{\grpd_F} \to 0.
\end{equation}
The reduced \editB{and essential analogues} of the above \editB{are} not exact in the middle.
In fact, following ideas of \cites{higson-lafforgue-skandalis,rufus}, this can be detected at the level of \(K\)-theory.
\begin{proposition} \label[proposition]{prop:k-0-failure-exactness}
  Let \(\grpd, V\) and \(F\) be as above. Neither
  \begin{enumerate}
    \item \(K_0(\redalg{\grpd_V}) \to K_0(\redalg{\grpd}) \to K_0(\redalg{\grpd_F})\) nor
    \item \(K_0(\essalg{\grpd_V}) \to K_0(\essalg{\grpd}) \to K_0(\redalg{\grpd_F})\)
  \end{enumerate}
  is exact in the middle.
\end{proposition}
\begin{proof}
  The argument here is essentially the same as the one in \cite{higson-lafforgue-skandalis}*{Proposition 3}.

  By \cref{cauchy-bundle-groups} the function \(\chi_B\) is a projection belonging to \(\redalg{\grpd}\).
  Recalling that \(B \coloneqq X \sqcup Y\) it follows that \(\rest_F^{\red}(\chi_B) = 0\), so \([\chi_B]_0\) does certainly belong to the kernel of the right map in the first sequence.
  It is thus enough to prove that it is not in the image of \(K_0(\redalg{\grpd_V}) \to K_0(\redalg{\grpd})\), but this is clear, for \(\lambda_{y_i}([\chi_B]_0)\) is never \(0\), whereas \(\lambda_{y_i}([\fa]_0)\) is \(0\) for co-finitely many \(i \in \NN\) whenever \(\fa \in \redalg{\grpd_V}\).

  The argument for the second item is fundamentally the same. Indeed, note that \(\chi_B + \singideal \in \essalg{\grpd}\) is a projection whose \(K_0\) class is never \(0\) on any \editB{\(\{x_{ij} : i, j \in \NN\}\)}.
  However, the elements in the image of \(K_0(\essalg{\grpd_V})\) are \(0\) on co-finitely many \(\{y_i : i \in \NN\}\).
\end{proof}

\begin{theorem} \label[theorem]{thm:baum-connes}
  If \(H\) has the Haagerup property then \(\grpd\) is a (non-Hausdorff) counterexample to the Baum--Connes Conjecture. In fact, neither \(\mu_\grpd^{\red}\) \eqref{eq:red-assembly-map} nor \(\mu_{\grpd}^{\ess}\) \eqref{eq:ess-assembly-map} are surjective. 
\end{theorem}
\begin{proof}
  Non-surjectivity of \(\mu_\grpd^{\red} \colon K^{\rm top}_*(\grpd) \to K_*(\redalg{\grpd})\) is an immediate consequence of \cite{higson-lafforgue-skandalis}*{Lemma 1 and Proposition 2} and the first item of \cref{prop:k-0-failure-exactness}.
  On the other hand, \(\mu_{\grpd}^{\ess} \colon K^{\rm top}_*(\grpd) \to K_*(\essalg{\grpd})\) is non-surjective essentially for the same reason, so we only sketch the proof.
  Note that \(\bar{X_i}\) is a closed subgroupoid of \(\grpd\), and that the diagram
  \begin{center}
    \begin{tikzcd}
      & K^{\rm top}_*\left(\grpd\right) \arrow[d]{}{\mu^{\max}_{\grpd}} \arrow[r]{}{} & K^{\rm top}_*\left(\grpd_F\right) \arrow[equal,d]{}{} \\
      K_*\left(\maxalg{\grpd_V}\right) \arrow[d]{}{} \arrow[r]{}{} & K_*\left(\maxalg{\grpd}\right) \arrow[d]{}{} \arrow[r]{}{} & K_*\left(\maxalg{\grpd_F}\right) \arrow[equal,d]{}{} \\
            K_*\left(\oplus_{n \in \NN}\cont{\bar{\NN}}\right) \arrow[r]{}{} & K_*\left(\essalg{\grpd}\right) \arrow[r]{}{} & K_*\left(\redalg{\grpd_F}\right)
    \end{tikzcd}
  \end{center}
  commutes (the equalities on the right column are a consequence of \editB{the Haagerup property of \(H\) and} \cite{tu-1999}*{Thm.~0.1}). Noting that \(\essalg{\grpd_V} = \oplus_{n \in \NN} \cont{\bar{\NN}}\), the claim follows from the second item in \cref{prop:k-0-failure-exactness} by trying (and failing) to trace \([\chi_B + \singideal]_0 \in K_0(\essalg{\grpd})\) back to \(K^{\rm top}_*(\grpd)\) in the previous diagram.
\end{proof}

Observe that \cite{higson-lafforgue-skandalis}*{Section 5} proves there are Hausdorff \'etale groupoids which do not satisfy the Baum--Connes Conjecture. Likewise, there are non-Hausdorff counterexamples in \cite{higson-lafforgue-skandalis}*{Sections 2,3 and 4}.
However, the counterexamples there either have no singular functions or \(\mu_{\grpd}^{\ess}\) is surjective.\footnote{\, The fact that \(\mu_{\grpd}^{\ess}\) is surjective for the ``\(\FF_2\)''-headed snake in \cite{higson-lafforgue-skandalis}*{Section 2, 3rd Counterexample} follows since \(\essalg{\grpd} = \cont{\bar{\NN}}\), and a routine \(6\)-term exact sequence argument applied to \(0 \to \contz{\NN} \to \cont{\bar{\NN}} \to \CC \to 0\).}
Thus, \cref{thm:baum-connes} slightly improves the results of \cite{higson-lafforgue-skandalis}.

\section{Second example: a minimal and effective groupoid} \label{sec:minimal-effective}

We now move on to the second and last example of the paper, which is a groupoid associated to a self-similar action in the sense of \cref{subsec:selfsim_prelim}. This is more involved than the one in \cref{sec:bundle-groups}, but follows the same basic strategy (outlined in \cref{sec:strategy}). 

Let $F \coloneqq F(a,b)$ be the free group on two generators $a$ and $b$, and let $H$ be any countable non-amenable group.
Let $G \coloneqq H \times F \times \ZZ \times \ZZ$, and let $K \coloneqq H \times F \times \ZZ$. For ease of notation, when it does not cause confusion, we identify the groups $H, F$ and $\ZZ^2$ with their natural images in $G$ and view them as subgroups. We define a (faithful) self-similar action of $G$ on the alphabet $\ZZ \sqcup \ZZ \sqcup K \sqcup K$.
\cref{fig:tree-drawings} is a pictorial depiction of the action. The precise definition below involves quite a bit of notation, we encourage the reader to think over how this corresponds to \cref{fig:tree-drawings}.

\begin{figure}[h]
  \includegraphics[width=0.45\linewidth]{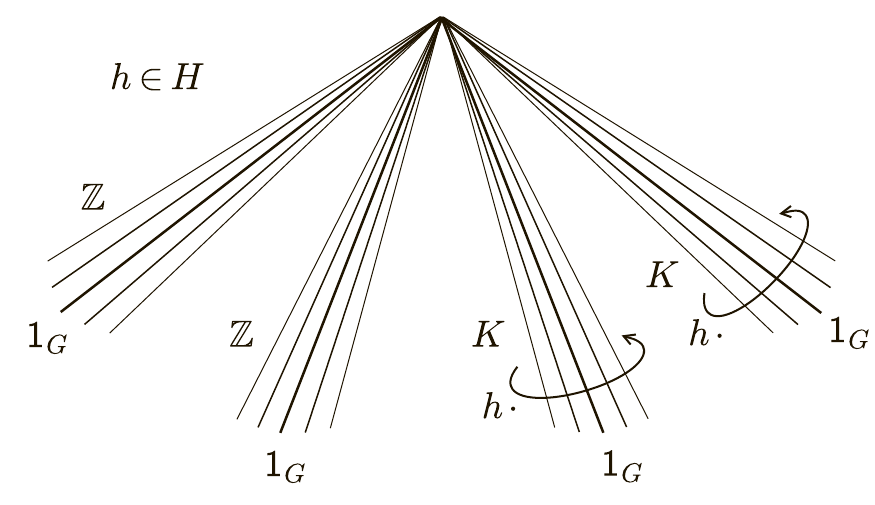}
  \begin{tabular}{cc}
    \includegraphics[width=0.45\linewidth]{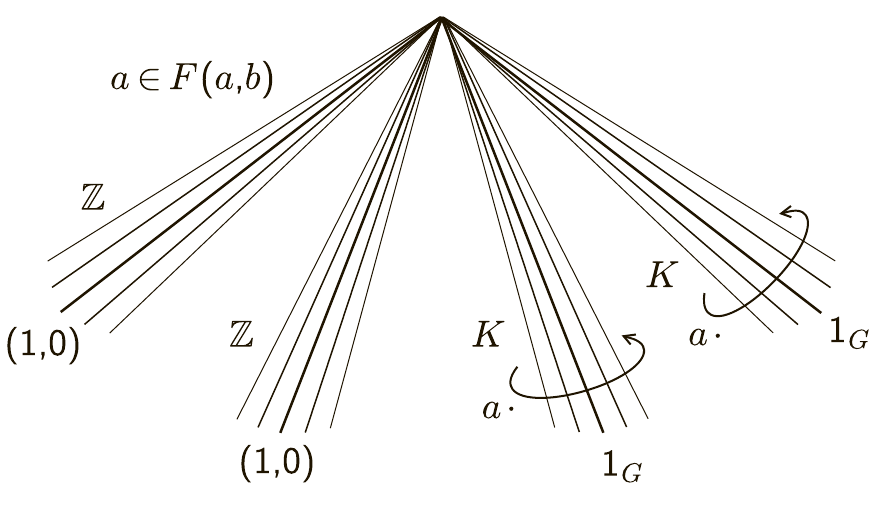} & \includegraphics[width=0.45\linewidth]{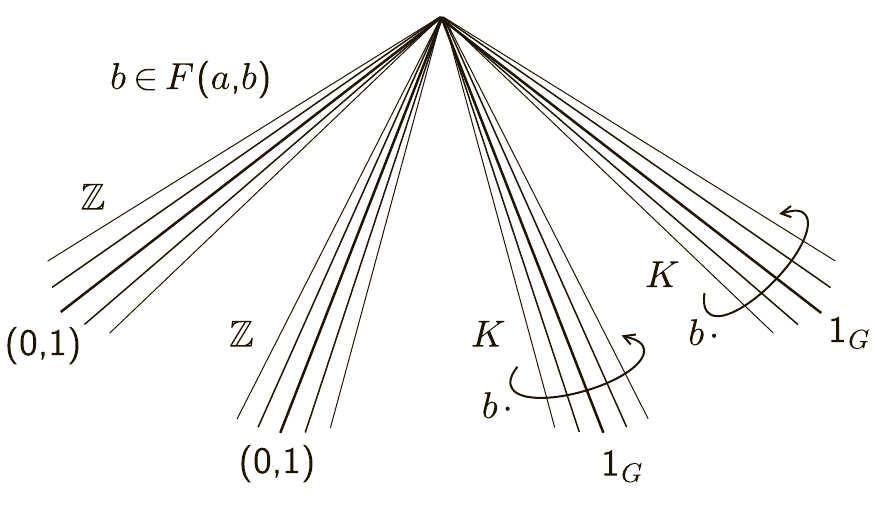}\\
    \includegraphics[width=0.45\linewidth]{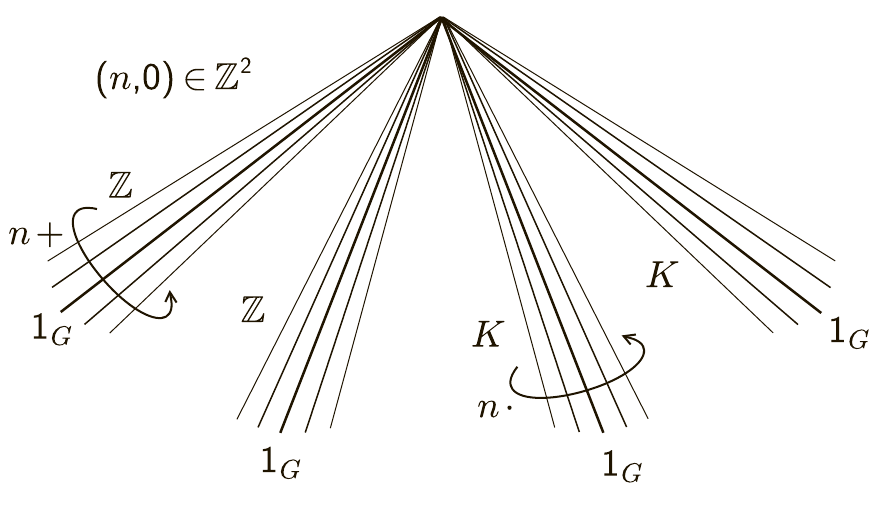} & \includegraphics[width=0.45\linewidth]{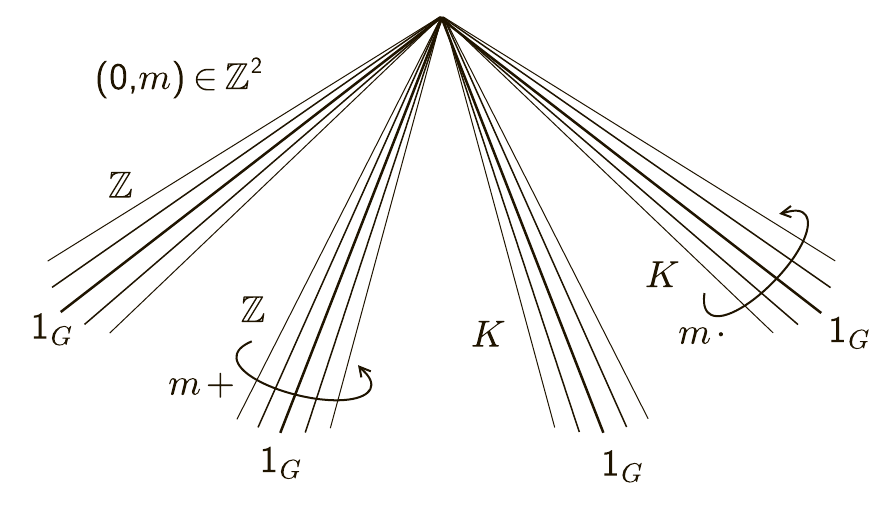}
  \end{tabular}
  \caption{The self-similar action of $G$, depicting how different elements of $G$ act on letters (viewed as the first level of the tree associated to finite words in \(X\)), and what their sections are at each letter.}
  \label{fig:tree-drawings}
\end{figure}  

Let $\tau \colon G \to G$ denote the morphism which maps
$$\begin{array}{l}
(h, a, n,m) \mapsto (1_H, 1_F, 1,0) \\
(h, b, n,m) \mapsto (1_H, 1_F, 0,1)
\end{array}$$
for any $h \in H$ and $(n,m) \in \ZZ^2$. In other words, $\tau$ is simply the abelianization map $F \to \ZZ^2$ extended to $G$.
Let furthermore $\pi_1 \colon G \to K$ be the obvious morphism with kernel $\{1_H\} \times \{1_F\} \times \{0\} \times \ZZ$ and $\pi_2 \colon G \to K$ the morphism with kernel $\{1_H\} \times \{1_F\} \times \ZZ \times \{0\}$.
Let $\zeta_1 \colon G \to \ZZ$ be the morphism with kernel $H \times F \times  \{0\} \times \ZZ$  and $\zeta_2\colon G \to \ZZ$ the morphism with kernel $H \times F \times \ZZ \times \{0\}$. These maps are depicted in
\cref{fig:maps}.

\begin{figure} 
  \begin{tikzpicture}[
    every node/.style={font=\Large},
    text height=1.5ex, text depth=.25ex,
    arrow/.style={-{Straight Barb}, thick},
    node distance=2.5cm and 3.5cm
    ]
    
    \node (G) at (0,6) {$G = H \times {F}(\textcolor{teal}{a}, \textcolor{orange}{b}) \times \textcolor{purple}{\mathbb{Z}} \times \textcolor{blue}{\mathbb{Z}}$};
    
    \node (H1) at (-4,4) {\raisebox{-3pt}{$K = H \times {F}(\textcolor{teal}{a}, \textcolor{orange}{b}) \times \textcolor{purple}{\mathbb{Z}}$}};
    \node (H2) at (4,4)  {\raisebox{-3pt}{$K = H \times {F}(\textcolor{teal}{a}, \textcolor{orange}{b}) \times \textcolor{blue}{\mathbb{Z}}$}};
    
    \node (Z1) at (-4,2) {\raisebox{-3pt}{$\mathbb{Z} = 1_H \times 1_F \times \textcolor{purple}{\mathbb{Z}}$}};
    \node (Z2) at (4,2)  {\raisebox{-3pt}{$\mathbb{Z} = 1_H \times 1_F \times \textcolor{blue}{\mathbb{Z}}$}};
    
    \node (Center) at (0,1) {$ 1_H  \times \textcolor{teal}{\mathbb{Z}} \times \textcolor{orange}{\mathbb{Z}} \times  (\textcolor{purple}{0},\textcolor{blue}{0})  \cong \textcolor{purple}{\mathbb{Z}} \times \textcolor{blue}{\mathbb{Z}} \leq G$};
    
    
    \draw[arrow] (G) -- node[above left=1pt] {$\pi_1$} (H1);
    \draw[arrow] (G) -- node[above right=1pt] {$\pi_2$} (H2);
    
    \draw[arrow] (G) to[bend left] node[left=2pt] {$\zeta_1$} (Z1);
    \draw[arrow] (G) to[bend right] node[right=2pt] {$\zeta_2$} (Z2);
    
    \draw[arrow] (G) -- node[right=2pt] {$\tau$} (Center);
    
    
  \end{tikzpicture}
  \caption{Maps used to define the self-similar action of $G$.}
  \label{fig:maps}
\end{figure}

Let $Y \coloneqq \{y_n^1: n \in \ZZ\} \cup \{y_n^2 : n \in \ZZ\}$, and $Z \coloneqq \{z_k^1: k \in K\} \cup \{z_k^2 : k \in K\}$.
We define the alphabet as $X \coloneqq Y \cup Z$ and denote the set of finite words over $X$ by $X^*$. The self-similar action of $G$ on $X^*$ is defined as follows:
\begin{equation} \label{def:self-similar-action}
  \begin{array}{ccc}
    g(y_n^1) \coloneqq y_{\zeta_1(g)+ n}^1; & \quad g(y_n^2) \coloneqq y_{\zeta_2(g)+ n}^2; & \quad g|_{y^1_n} = g|_{y^2_n} \coloneqq \tau(g); \vspace{0.5em}\\
    g(z_k^1) \coloneqq z_{\pi_1(g)\cdot k}^1; & \quad g(z_k^2) \coloneqq z_{\pi_2(g)\cdot k}^2; & \quad g|_{z^1_k}=g|_{z^2_k} \coloneqq 1_G. 
  \end{array}
\end{equation}
In particular, the action of $G$ is faithful even on $Z$, as $\ker \pi_1 \cap \ker \pi_2 = \{1_G\}$, and hence also on $X^*$. As usual, we extend the self-similar action of $G$ on $X^*$ to the set of (right) infinite words $X^\omega$ in the obvious way.
Recall that we say that $g$ \emph{strongly fixes} a word $\gamma \in X^*$ if $g(\gamma)=\gamma$ and $g|_\gamma=1_G$. Let $\grpd$ be the ample groupoid associated to the self-similar action as defined in \cref{subsec:selfsim_prelim}. Recall that two germs $[s, w]$, $[t, w]$ of $\grpd$ are equal if there is some prefix $\gamma$ of $w$ with $s\gamma=t\gamma$.

\begin{proposition}
The groupoid \(\grpd\) is minimal and effective.
\end{proposition}

\begin{proof}
By \cite{BenNora21}*{Proposition 6.1} it suffices to show that whenever $g \in G$ strongly fixes a cofinite set of letters in $X$, it is necessarily the identity. Observe that if $g \neq 1_G$ then at least one of $\pi_1(g)$ or $\pi_2(g)$ is not $1_H$, hence $g$ does not fix any letter in at least one of the infinite sets $\{z_k^1: k \in K\}$ or $\{z_k^2: k \in K\}$.
\end{proof}

\begin{theorem} \label[theorem]{thm:alg-sing-minimal-effective}
The algebraic singular ideal $\asingideal = \singideal \cap \algalg{\grpd}$ is not dense in the singular ideal $\singideal$.
In particular, neither is $\singideal \cap \steinalg{\grpd}$ dense in $J$.
\end{theorem}

We prove \cref{thm:alg-sing-minimal-effective} following the steps outlined in \cref{sec:strategy}. 

\subsection{Scattering.}
Since $H$ is a non-amenable group, by \cref{prop:nonamenable-groups-can-scatter} we can find some sequence $\{H_n\}_{n \in \NN}$ of finite non-empty subsets of $H$ such that the sequence $\frac{1}{\abs{H_n}}\sum_{h \in H_n} h$ tends to $0$ in $\redalg H$.
Identifying $H$ with its image under the natural embedding $H \subseteq G$, we may consider $\{H_n\}_{n \in \NN}$ as a sequence of subsets in $G$.
Now consider the sets of compact open bisections 
\begin{equation} \label{mini-eff-def-un}
  \CU_n \coloneqq \left\{\left(h, D\left(\epsilon\right)\right) : h \in H_n\right\}
\end{equation}
of $\grpd$. As \(D(\epsilon) = \grpd^{(0)}\), every bisection in every \(\CU_n\) has full source and range.

\begin{lemma} \label[lemma]{min-eff-h-intersects-in-B}
For any $h_1 \neq h_2$, we have $(h_1, D(\epsilon)) \cap (h_2, D(\epsilon))=\displaystyle{\bigcup_{y \in Y}} D(y)$.
\end{lemma}
\begin{proof}
By definition, the intersection $(h_1, D(\epsilon)) \cap (h_2, D(\epsilon))$ consists of germs $[h_1,w]$ where 
$$h_1(\gamma)h_1|_\gamma=h_1\gamma=h_2\gamma=h_2(\gamma)h_2|_\gamma$$
for some prefix $\gamma$ of $w$. Since $H$ acts freely on $Z$, $h_1\gamma=h_2\gamma$ implies that $\gamma$ (and therefore $w$) begins with $y$. Conversely for any $y \in Y$,  we have $h_iy=h_i(y)h_i|y=y$ $(i=1,2)$, hence $[h_1,w]=[h_2,w]=[1, w]$ whenever $w$ begins with $y$. It follows that $[h_1,w]=[h_2,w]$ if and only if $w \in {\bigcup_{y \in Y}} D(y)$, in which case $[h_1,w]=[h_2,w]=[1,w]=w$.
\end{proof}

Define
\begin{equation} \label{mini-eff-def-bn}
  B \coloneqq \displaystyle{\bigcup_{y \in Y}} D(y) \subseteq \grpd^{(0)}; \; C \coloneqq \displaystyle{\bigcup_{z \in Z}} D(z) \subseteq \grpd^{(0)}; \; \text{ and } \; \fb_n \coloneqq \frac{1}{\abs{\CU_n}} \sum_{U \in \CU_n} \chi_{U} = \frac{1}{\abs{H_n}} \sum_{h \in H_n} \chi_{\left(h, D\left(\epsilon\right)\right)}.
\end{equation}
The following is the analogue of \cref{cauchy-bundle-groups} in this setting.
\begin{proposition} \label[proposition]{cauchy-minimal-effective}
The sequence $\{\fb_n\}_{n \in \NN}$ is Cauchy in $\redalg\grpd$. In fact, it converges to \(\fb \coloneqq \chi_B\).
\end{proposition}

The proof of \cref{cauchy-minimal-effective} is somewhat long, so we divide it into smaller, more digestible pieces.
In the following, whenever we write \([\alpha g \beta^*, w] \in \grpd\) we are implicitly assuming that the germ is defined, that is, that \(\beta \in X^*\) is a finite prefix of the (possibly infinite) word \(w \in \grpd^{(0)}\).

\begin{lemma} \label[lemma]{lemma:y-z-epsilon-well-behaved-under-h}
  Let $w \in \grpd^{(0)}$ and consider the sets $C \grpd w, B \grpd w,  \epsilon \grpd w \subseteq \grpd w$. The following assertions hold.
  \begin{enumerate}
    \item \label{lemma:y-z-epsilon-well-behaved-under-h:1} \(C \grpd w, B \grpd w\) and \(\epsilon \grpd w\) partition \(\grpd w\).
    \item \label{lemma:y-z-epsilon-well-behaved-under-h:2} \(\epsilon \grpd w\) is non-empty if and only if \(w\) is a finite word.
    \item \label{lemma:y-z-epsilon-well-behaved-under-h:3} 
    The sets \(B \grpd w, C \grpd w\) and \(\epsilon \grpd w\) are invariant under the \editB{left} action of $(h, D(\epsilon))$
    for all \(h \in H \subseteq G\).
    \item \label{lemma:D(z)} Any arrow of $D(z)\grpd w$ can be expressed in the form $[z\alpha g \beta^*, w]$ with $z \in Z$.
  \end{enumerate}
\end{lemma}
\begin{proof}
  The proof of (\ref{lemma:y-z-epsilon-well-behaved-under-h:1}) immediately follows from the fact that \(\grpd^{(0)} = \{\epsilon\} \sqcup B \sqcup C\). 

  For (\ref{lemma:y-z-epsilon-well-behaved-under-h:2}), note that
  if \(w \in X^*\) is a finite word then \([1_G w^*, w] \in \epsilon \grpd w\),
  and conversely $X^\omega$ is invariant under the action of $S$, so \(\epsilon \grpd w\) is empty whenever $w$ is infinite.

  Let us continue onto (\ref{lemma:y-z-epsilon-well-behaved-under-h:3}). Fix let \(h \in H\), and let $[\alpha g \beta^*, w] \in \grpd w$ be given. Observe that $r([\alpha g \beta^*, w])=\alpha g(\beta^*w)$, furthermore
  \[
    \left(h, D\left(\epsilon\right)\right) \, \left[\alpha g \beta^*, w\right] 
    = [h\cdot\alpha g \beta^*, w]=
     \left[h\left(\alpha\right) h|_{\alpha} g \beta^*, w\right],
  \]
  where $r([h\left(\alpha\right) h|_{\alpha} g \beta^*, w])=h(\alpha)h|_{\alpha} g (\beta^* w)$.
  Since the actions of $g$ and $h$ are length-preserving, we have 
  $$\alpha g(\beta^*w)=\epsilon \iff \alpha=\beta^*w=\epsilon \iff h(\alpha)g(\beta^*w)=\epsilon,$$
  which shows the invariance of $\epsilon \grpd w$. Secondly, we have $\alpha g(\beta^*w) \in B$ if either $\alpha=\epsilon$ and $g(\beta^*w) \in B$, in which case $h(\alpha)g(\beta^*w)= g(\beta^*w)\in B$ as well; or $\alpha \in B$, in which case $h(\alpha) \in B$ as well, since $h$ fixes the first letter of any nonempty word whose first letter is in \(B\). This shows the invariance of $B$. The invariance of $C$ follows from the fact that \(h\) (or, equivalently, \((h, D(\epsilon))\)) defines a homeomorphism of \(\grpd^{(0)} = \{\epsilon\} \sqcup B \sqcup C\).
  
  For the last part, let $[\gamma g \delta^*, w] \in D(z)\grpd w$, that is, suppose $r([\gamma g \delta, w])=\gamma g(\delta^* w) \in D(z)$. Then either $\gamma \in D(z)$, in which case $[\gamma g \delta^*, w]$ is already in the required form, or $\gamma=\epsilon$ and $g(\delta^* w) \in D(z)$, in which case $w=\delta x w'$ where $g(x)=z$. Then 
  $$g\delta^*\delta x=zg|_x=zg|_xx^*\delta^*\delta x,$$
  therefore $[g\delta^*, w]=[zg|_xx^*\delta^*, w]$.
\end{proof}

Recall the definition of $\lambda_w$ from \cref{sec:prelim}, and note that by \cite{BenNora21}*{Thm. 2.11} there is an embedding $\CC (H) \to \steinalg\grpd \subseteq \algalg\grpd$, $h \mapsto \chi_{(h, D(\epsilon))}$, which allows us to consider $\CC (H)$ as a subalgebra of $\redalg \grpd$.
A consequence of \cref{lemma:y-z-epsilon-well-behaved-under-h} is the following.
\begin{lemma} \label[lemma]{lemma:y-z-epsilon-well-behaved-under-h-projections}
  Let \(p_{Y,w}\), $p_{Z,w}$, $p_{\epsilon,w}$  be the orthogonal projections from $\ell^2(\grpd w)$ to $\ell^2(B \grpd w)$, $\ell^2(C \grpd w)$, $\ell^2(\epsilon \grpd w)$ respectively. Then these are mutually orthogonal projections in $\CB(\ell^2(\grpd w))$ such that
  \begin{enumerate}
    \item \label{lemma:y-z-epsilon-well-behaved-under-h-projections:1} \(p_{Y,w} + p_{Z,w} + p_{\epsilon,w}\) is the identity of \(\CB(\ell^2(\grpd w))\);
    \item \label{lemma:y-z-epsilon-well-behaved-under-h-projections:2} \(p_{\epsilon,w} \neq 0\) if and only if \(w\) is a finite word; and
    \item \label{lemma:y-z-epsilon-well-behaved-under-h-projections:3} \(\lambda_w(\chi_{(h, D(\epsilon))})\) commutes with \(p_{Y,w}, p_{Z,w}\) and \(p_{\epsilon,w}\) for all \(h \in H\).
  \end{enumerate}
\end{lemma}

Recalling \eqref{mini-eff-def-bn}, it follows from \cref{lemma:y-z-epsilon-well-behaved-under-h-projections} that 
\begin{equation} \label{eq:lambda-w-fb-n-decomposes}
  \lambda_w\left(\fb_n\right) = \lambda_w\left(\fb_n\right) p_{Y,w} + \lambda_w\left(\fb_n\right) p_{Z,w} + \lambda_w\left(\fb_n\right) p_{\epsilon,w}.
\end{equation}
This allows us to decompose \(\lambda_w\) into sub-representations (\emph{only}) on the part of $\algalg\grpd$ coming from \(\CC (H)\) (which is sufficient for the scattering).
Consider the map
\[
  \pi_{Y,w} \colon H \to \CB\left(\ell^2\left(B \grpd w\right)\right), \; \text{ where } \; \pi_{Y,w}\left(h\right) \coloneqq \lambda_w\left(\chi_{\left(h,D(\epsilon)\right)}\right) \, p_{Y,w}.
\]
Similarly, we may define \(\pi_{Z,w}(h) \coloneqq \lambda_w(\chi_{(h,D(\epsilon))}) p_{Z,w}\) and \(\pi_{\epsilon,w}(h) \coloneqq \lambda_w(\chi_{(h,D(\epsilon))}) p_{\epsilon,w}\).

\begin{lemma} \label[lemma]{min-eff-norm-bound-above}
  The maps \(\pi_{Y,w}, \pi_{Z,w}\) and \(\pi_{\epsilon,w}\) are representations of \(H\). Moreover, the following hold.
  \begin{enumerate}
    \item \label{min-eff-norm-bound-above:1} \(\pi_{Y,w}(h) = 1 \in \CB(\ell^2(B \grpd w))\) for all \(h \in H\); and
    \item \label{min-eff-norm-bound-above:2} \(\pi_{Z,w}\) and \(\pi_{\epsilon,w}\) are weakly contained in \(\lambda_H \colon H \to \CB(\ell^2(H))\) (the left regular representation of \(H\)).
  \end{enumerate}
  In particular,
  \[
    \norm{\ff}_{\redalg{\grpd}} \stackrel{\rm (def)}{=} \sup_{w \in \grpd^{(0)}} \norm{\lambda_w\left(\ff\right)} \leq \max \left\{\norm{\pitriv (\ff)}, \norm{\lambda_H(\ff)}\right\}
  \]
  for all \(\ff \in \stein H \subseteq \redalg{\grpd}\).
\end{lemma}
\begin{proof}
  The fact that \(\pi_{Y,w}, \pi_{Z,w}\) and \(\pi_{\epsilon,w}\) are all representations follows from the fact that
  $h \mapsto \chi_{(h, D(\epsilon))}$ embeds $\stein{H}$ into $\algalg\grpd$, that
   \(\lambda_w\) is a representation and that \(\lambda_w(\chi_{(h, D(\epsilon))})\) commutes with \(p_{Y,w}, p_{Z,w}\) and \(p_{\epsilon,w}\) (see \cref{lemma:y-z-epsilon-well-behaved-under-h-projections}).

  In order to prove (\ref{min-eff-norm-bound-above:1}), let \([\alpha g \beta^*, w] \in B \grpd w\) be given: in particular $r([\alpha g \beta^*, w])=yw'$ for some $y \in Y$ and $w'=X^\ast \cup X^\omega$.
  Then
  \[
    \pi_{w,Y}\left(h\right) \delta_{\left[ \alpha g \beta^*, w\right]} \stackrel{\rm (def)}{=} \lambda_w\left(\chi_{\left(h, D\left(\epsilon\right)\right)}\right) \delta_{\left[ \alpha g \beta^*, w\right]} = \delta_{\left(h, D\left(\epsilon\right)\right) \left[ \alpha g \beta^{*}, w\right]} = 
    \delta_{[h, yw']\left[\alpha g\beta^*, w\right]}.
  \]
  Now note that $hy=h(y)h|_y=y$, so $[h,yw']=[1,yw']$, in particular $[h, yw']\left[\alpha g\beta^*, w\right]=[\alpha g \beta^*, w]$.

  Item (\ref{min-eff-norm-bound-above:2}) follows from analogous ideas, so let us first consider the case of \(\pi_{\epsilon,w}\).
  We claim that \(\pi_{\epsilon,w}\) is either \(0\) or unitarily equivalent to the left regular representation \(\lambda_H \colon H \to \CB(\ell^2(H))\).
  In fact, by \cref{lemma:y-z-epsilon-well-behaved-under-h-projections} we may assume that \(w \in X^*\) is a finite word, for it is was not then \(p_{\epsilon,w} = 0\), and thus \(\pi_{\epsilon, w} = 0\) would indeed be weakly contained anywhere.
  Observe that \(\epsilon \grpd \epsilon = \{[g, \epsilon] : g \in G\}\), and that these elements are pairwise different, since $g\epsilon=g$ for all $g \in G$.
  In particular it follows that there is a bijection \(\epsilon \grpd w \ni [g w^*, w] \mapsto g \in G\) which, in turn, induces a unitary operator \(U_{\epsilon,w} \colon \ell^2(\epsilon \grpd w) \to \ell^2(G)\) intertwining the representation \(\pi_{\epsilon,w}\) and \(\lambda_H\) (which we see as a sub-representation of \(\lambda_G\)).
  Indeed, the diagram
  \begin{center}
    \begin{tikzcd}
      \ell^2\left(\epsilon \grpd w\right) \arrow[r]{}{U_{\epsilon,w}} \arrow[d]{}{\pi_{\epsilon,w}\left(h\right)} & \ell^2\left(G\right) \arrow[d]{}{\lambda_H\left(h\right)} \\
      \ell^2\left(\epsilon \grpd w\right) \arrow[r]{}{U_{\epsilon,w}} & \ell^2\left(G\right)
    \end{tikzcd}
  \end{center}
  commutes for all \(h \in H\), for \(U_{\epsilon,w} \pi_{\epsilon,w}(h) \delta_{[gw^*,w]} =
  U_{\epsilon,w}\delta_{[hgw^*,w]}= \delta_{hg} = \lambda_H(h) \delta_g=\lambda_H U_{\epsilon,w}\delta_{[gw^*,w]}\).

  The case for \(\pi_{Z,w}\) is more involved. Recall that \(K = H \times F \times \ZZ\), and that \(Z = \{z_k^1 : k \in K\} \cup \{z_k^2 : k \in K\}\).
  Also, see \eqref{def:self-similar-action} (cf.\  \cref{fig:tree-drawings}) for the definition of the action of \(H\) on \(Z\), which is free:
  \begin{equation} \label{eq:how h acts on z}
    h\left(z_k^1\right) = z_{h k}^1; \quad h\left(z_k^2\right) = z_{h k}^2; \quad h|_{z^1_k} = h|_{z^2_k} = 1_H = 1_G 
  \end{equation}
  for all \(h \in H \subseteq G\) (since \(\pi_i(h) = h\), see \cref{fig:maps}).
  Moreover, observe that \(p_{Z,w} = \sum_{z \in Z} p_{z,w}\) (where the sum is in the strong operator topology), and \(p_{z,w}\) is defined to be the orthogonal projection from \(\ell^2(C \grpd w)\) onto \(\ell^2(D(z) \grpd w)\) (similarly to before).
  These are mutually orthogonal projections, since they correspond to the sets \(D(z) \grpd w\), which partition \(C \grpd w\), for \(C = \bigcup_{z \in Z} D(z)\) by definition.
  Recall that $D(z)$ consists of arrows of the form $[z\alpha g\beta^*,w]$ by \cref{lemma:y-z-epsilon-well-behaved-under-h} (\ref{lemma:D(z)}).
  In addition, it follows from \eqref{eq:how h acts on z} that
  \[
    \lambda_w\left(\chi_{\left(h, D\left(\epsilon\right)\right)}\right) \, p_{z^i_k,w} = p_{z^i_{hk},w} \, \lambda_{w}\left(\chi_{\left(h, D\left(\epsilon\right)\right)}\right)
  \]
    and
  \[
  \lambda_w\left(\chi_{\left(h, D\left(\epsilon\right)\right)}\right) \, \delta_{\left[z_k^i \alpha g \beta^*, w\right]} = \delta_{\left[h z_k^i \alpha g \beta^*, w\right]} = \delta_{\left[h\left(z_k^i\right) \alpha g \beta^*, w\right]} = \delta_{\left[z_{hk}^i \alpha g \beta^*, w\right]}
  \]
  for all \(z^i_k \in Z\), $\left[z_{k}^i \alpha g \beta^*, w\right] \in D(z^i_k)\grpd w$ and \(h \in H\).
  
  Now observe that $[z_1\alpha_1 g_1 \beta_1^*, w]=[z_2\alpha_2 g_2 \beta_2^*, w]$ if and only if $z_1=z_2$ and $[\alpha_1 g_1 \beta_1^*, w]=[\alpha_2 g_2 \beta_2^*, w]$. It follows that there is a bijection 
  \(C \grpd w \ni [z_k^i \alpha g \beta^*, w] \mapsto (k, i, [\alpha g\beta^*,w]) \in K \times \ZZ_2 \times \grpd w,\)
  which in turn induces a unitary \(U\) making the diagram
  \begin{center}
    \begin{tikzcd}
      \ell^2\left(C \grpd w\right) \arrow[r]{}{U} \arrow[d]{}{\pi_{Z,w}\left(h\right)} & \ell^2\left(K\right) \otimes \ell^2\left(\ZZ_2\right) \otimes \ell^2\left(\grpd w\right) \arrow[d]{}{\lambda_K\left(h\right) \otimes \pitriv\left(h\right)} \\
      \ell^2\left(C \grpd w\right) \arrow[r]{}{U} & \ell^2\left(K\right) \otimes \ell^2\left(\ZZ_2\right) \otimes \ell^2\left(\grpd w\right)
    \end{tikzcd}
  \end{center}
  commute for all \(h \in H\).
  It follows that \(\pi_{Z,w}\) is unitarily equivalent to a subrepresentation of \(\lambda_K \otimes \pitriv\), which is weakly contained in \(\lambda_K\) by Fell's absorption principle.

  For the last statement of the lemma, note that it follows from (\ref{min-eff-norm-bound-above:1}) and (\ref{min-eff-norm-bound-above:2}) that
  \begin{equation} \label{eq:min-eff:bounds-norms}
    \norm{\pi_{Y,w}\left(\ff\right)} \stackrel{(\ref{min-eff-norm-bound-above:1})}{=} \norm{\pitriv\left(\ff\right)}; \quad \norm{\pi_{Z,w}\left(\ff\right)} \stackrel{(\ref{min-eff-norm-bound-above:2})}{\leq} \norm{\lambda_H\left(\ff\right)}; \quad \text{ and } \quad \norm{\pi_{\epsilon,w}\left(\ff\right)} \stackrel{(\ref{min-eff-norm-bound-above:2})}{\leq} \norm{\lambda_H\left(\ff\right)}
  \end{equation}
  for all \(\ff \in \stein H\) and \(w \in \grpd^{(0)}\).
  Thus, the ``in particular'' statement follows routinely:
  \begin{align*}
    \norm{\ff}_{\redalg{\grpd}} & \stackrel{\rm (def)}{=} \sup_{w \in \grpd^{(0)}} \norm{\lambda_w\left(\ff\right)} \stackrel{\rm \eqref{eq:lambda-w-fb-n-decomposes}}{=} \sup_{w \in \grpd^{(0)}} \norm{\lambda_w\left(\ff\right) \left(p_{Y,w} + p_{Z,w} + p_{\epsilon,w}\right)} \\
    & \stackrel{\rm (\cref{lemma:y-z-epsilon-well-behaved-under-h-projections})}{=} \, \max_{\star \in \left\{Y, Z, \epsilon\right\}} \; \sup_{w \in \grpd^{(0)}} \norm{\pi_{\star,w}\left(\ff\right)} \stackrel{\rm \eqref{eq:min-eff:bounds-norms}}{\leq} \max \left\{\norm{\pitriv (\ff)}, \norm{\lambda_H(\ff)}\right\}
  \end{align*}
  for all \(\ff \in \stein H \subseteq \redalg{\grpd}\).
\end{proof}

\begin{proof}[Proof of \cref{cauchy-minimal-effective}]
  We must prove the sequence \(\{\fb_n\}_{n \in \NN}\) is Cauchy in the reduced norm, that is,
  \begin{equation} \label{eq:min-effe-have-to-show}
    \norm{\fb_n - \fb_m} = \sup_{w \in \grpd^{(0)}} \norm{\lambda_w\left(\fb_n-\fb_m\right)} \xrightarrow{n, m \to \infty} 0.
  \end{equation}
  If we define
  \[
    \ff_n \coloneqq \frac{1}{\abs{\CU_n}} \sum_{h \in H_n} h \in \stein H
  \]
  for all \(n \in \NN\), then the embedding of $\CC (H) \to \redalg \grpd$ via $h \mapsto \chi_{(h,D(\epsilon))}$ sends $\ff_n$ to $\fb_n$ by \eqref{mini-eff-def-bn}, 
   and the ``in particular'' statement in \cref{min-eff-norm-bound-above} yields that 
  \[
    \norm{\fb_n - \fb_m} \leq \max \left\{\norm{\pitriv (\ff_n - \ff_m)}, \norm{\lambda_H(\ff_n - \ff_m)}\right\}.
  \]
  Note that \(\pitriv(\ff_n) = 1\) for all \(n \in \NN\). Moreover, \(\norm{\ff_n}_{\redalg H} \to 0\) by \cref{prop:nonamenable-groups-can-scatter}. In particular, \eqref{eq:min-effe-have-to-show} follows, so \(\{\fb_n\}_{n \in \NN}\) indeed converges to some \(\fb \in \redalg{\grpd}\).
  Thus, by \cref{cauchy-plus-convergent}, the proof will be complete once we show that \(\fb_n\) converges to \(\chi_B\) in the supremum norm. This now follows from \cref{min-eff-h-intersects-in-B} and Step 1 of \cref{sec:strategy}.
\end{proof}

\subsection{Finding a \texorpdfstring{$B$}{B}-singular element.}
The $B$-singular element will be defined with the help of $F = F(a,b)$, which we identify with its image under the natural embedding $F \hookrightarrow G$. Let
\[
  \fa \coloneqq \chi_{(1, D(\epsilon))}-\chi_{(a, D(\epsilon))}-\chi_{(b, D(\epsilon))}+\chi_{(ab, D(\epsilon))} \in \stein\grpd.
\]

\begin{lemma} \label[lemma]{lemma:min-eff-b-singular}
The function $\fa$ is $B$-singular, that is, $\fa \ast \chi_B$ is singular.
\end{lemma}

\begin{proof}
We describe $\supp{\fa \ast \chi_B}$. Observe that $\fa \ast \chi_B=\chi_{(1, B)}-\chi_{(a, B)}-\chi_{(b, B)}+\chi_{(ab, B)}$, so all we need to understand is where these (non-compact) open bisections overlap, that is, which germs of the form $[1,yw], [a,yw], [b,yw], [ab,yw]$ coincide for $y \in Y$ and $w \in X^\ast \cup X^\omega$.
Recall that $[s,w]=[t,w]$ if and only if $s\gamma=t\gamma$ holds in $S$ for some prefix $\gamma$ of $w$.

For any $y \in Y$, by \eqref{def:self-similar-action} we have the following equalities in $S$, where as usual we identify $\ZZ \times \ZZ$ with its image in $G$:
\begin{equation}
\label{eq:product-with-y}
1y=y1,\ ay=y(1,0), \ by=y(0,1), \text{ and } aby=y(1,1).
\end{equation}
In particular, the germs $[1,y], [a,y], [b,y], [ab,y]$ are pairwise distinct and hence these germs are all in the support of $\fa \ast \chi_B$.

Now to see what happens for germs at longer words starting with $y$, we look at four cases depending on the second letter. Observe that for any $n \in \ZZ$ and $k \in K$, we have:
\begin{equation}
\label{eq:product-after-y}
\begin{array}{cc}
1y^1_n=y^1_n=(0,1)y^1_n; & \quad (1,0)y^1_n=y^1_{n+1}=(1,1)y^1_n; \vspace{0.5em}\\
1y^2_n=y^2_n=(1,0)y^2_n; & \quad (0,1)y^2_n=y^2_{n+1}=(1,1)y^2_n; \vspace{0.5em}\\
1z^1_k=z^1_k=(0,1)z^1_k; & \quad (1,0)z^1_n=z^1_{(1_H, 1_F, 1)\cdot k}=(1,1)z^1_k; \vspace{0.5em}\\
1z^2_k=z^2_k=(1,0)z^2_k; & \quad (0,1)z^2_n=z^2_{(1_H, 1_F, 1)\cdot k}=(1,1)z^2_k.  
\end{array}
\end{equation}
Combining \eqref{eq:product-with-y} and \eqref{eq:product-after-y} it follows that for any $y \in Y$ and $x \in X$, we have
$$\{1yx, abyx\}=\{ayx, byx\}.$$
In particular, for any non-empty (finite or infinite) word $w$, we have
$$\{[1,yw], [ab,yw]\}=\{[a,yw], [b,yw]\},$$
and as such $\fa \ast \chi_B$ cancels on all of these germs. It follows that 
$$\supp{\fa \ast \chi_B}=\{[1,y], [a,y], [b,y], [ab,y]\colon y\in Y\},$$
which has empty interior indeed.
\end{proof}


We have $\fa * \chi_B = \fa * \fb \in \singideal$; what remains to show is that it is not approximated by algebraic singular functions.

\subsection{ Showing \texorpdfstring{$\fa * \fb$}{a*b} is not in the closure of \texorpdfstring{$\asingideal$}{Jalg}}
Following the strategy outlined in \cref{sec:strategy}, we prove that if $\ff \in \steinalg \grpd$ is such that 
$\supp {\fa * \fb} \subseteq \supp \ff$, then $\ff \notin \singideal$. 
Recall that $\supp{\fa \ast \fb}=\{[1,y], [a,y], [b,y], [ab,y]\colon y\in Y\}$ (as is shown in the proof \cref{lemma:min-eff-b-singular}).
Therefore, in particular, $Y \subseteq \supp \ff$. \editB{\cref{thm:alg-sing-minimal-effective} then follows by \cite[Thm. 6.3]{GH2025}.}

\editB{In the following, whenever we say ``almost all'' we mean ``for all but, at most, finitely many''.}

\begin{lemma} \label[lemma]{lem:selfsim-stein}
Let $\ff \in \stein \grpd$ be such that $y \in \supp \ff$ for almost all $y \in Y$.
Then there exists $h \in G$ such that the compact open bisection $(h, D(z^1_k)) \subseteq \supp \ff$ for almost all $k \in K$.

In particular, if $\ff \in \steinalg \grpd$ is such that  $\supp {\fa * \fb} \subseteq \supp \ff$, then $\ff \notin \singideal$.
\end{lemma}
\begin{proof}
Suppose $\ff$ satisfies the conditions of the lemma, and by \cref{prop:stein description} write
$$\ff=\sum_{s \in S\setminus \{0\}} c_s \chi_{(s, D(s^*s))},$$
where only finitely many coefficients are nonzero. (Recall that if $s=\alpha g \beta^*$, then $D(s^*s)=D(\beta)$.)
We call an integer $n \in \ZZ$ \emph{typical} if for $i=1,2$ we have $y_n^i\in \supp \ff$ and, furthermore, $c_s=0$ for any $s=\alpha g\beta^*$ where $\beta$ begins with $y_n^i$. Observe that almost all indices in $\ZZ$ are typical by definition.

Let $n \in \ZZ$ be a typical index. For $i=1,2$, we have
\begin{equation}
\label{eq:nonzero1}
0 \neq \ff([1, y_n^i]) = \sum_{[s,y_n^i]=[1, y_n^i]} c_s,
\end{equation}
where the equality follows from \cite{BenNora21}*{Corollary 2.12} (or can be verified by definition). We remark that the condition $[s,y_n^i]=[1, y_n^i]$ is of course understood to imply that $[s,y_n^i]$ is defined, \emph{i.e.}\ that $y_n^i \in D(s^*s)$.
For $s=\alpha g\beta^*$, we have $y_n^i \in D(s^*s)=D(\beta)$ exactly if either $\beta=\epsilon$ or $\beta=y_n^i$ -- but in the latter case, $c_s=0$ since $n$ is typical. It follows that 
if $c_s \neq 0$, then 
\begin{equation}
\label{eq:vivalent-conds}
[s,y_n^i]=[1, y_n^i] \iff s y_n^i=y_n^i \iff s=g \in G \hbox{ with }\tau(g)=1_G \hbox{ and } \zeta_i(g)=0.
\end{equation}
Let
$G_i=\ker \tau \cap \ker \zeta_i \normal G$, and let 
$Z_i=\ker \pi_i \normal G$. Observe (see \cref{fig:maps}) that $Z_i \subseteq G_i$.
From \eqref{eq:nonzero1} and \eqref{eq:vivalent-conds} applied to the case when $i=1$, we obtain
\begin{equation*}
0 \neq \sum_{g \in G_1} c_g=\sum_{\Delta \in G_1/Z_1}\sum_{g \in \Delta} c_g,
\end{equation*}
in particular there must be at least one $\Delta \in G_1/Z_1 \leq G/Z_1$ such that $\sum_{g \in \Delta} c_g \neq 0$. Put $\Delta=hZ_1$ for some $h \in G_1$. Then for $g \in G$, we have $g \in hZ_1$ if and only if $\pi_1(g)=\pi_1(h)$, and so
\begin{equation}
\label{eq:nonzero2}
0 \neq \sum_{g \in hZ_1} c_g= \sum_{\pi_1(g)=\pi_1(h)} c_g.
\end{equation}
\editB{We call $k \in K$ \emph{typical} if for $i=1,2$, we have $c_s=0$ for any $s=\alpha g\beta^*$ where $\beta$ begins with $z_k^i$. Again, almost all indices in $K$ are typical.}
Now fix a typical index $k \in K$ and let $w \in X^* \cup X^\omega$, and consider
\begin{equation*}
  \ff\left(\left[h, z_k^1w\right]\right) = \sum_{[s,z_k^1w]=[h, z_k^1w]} c_s.
\end{equation*}
We aim to show that this sum is nonzero. If $c_s\neq 0$, then as before, the typicality of $k$ implies that $z_k^1w \in D(ss^*)$ holds if and only if $s$ is of the form $\alpha g$. For such $s$, we have  $[\alpha g,z_k^1w]=[h, z_k^1w]$ if and only if $\alpha g z_k^1\gamma=hz_k^1\gamma$ for some prefix $\gamma$ of $w$, which can only hold if $\alpha=\epsilon$. In this case, 
\begin{equation*}
  g z_k^1\gamma=z^1_{\pi_1(g)\cdot k}\gamma;\ \ hz_k^1\gamma=z^1_{\pi_1(h)\cdot k}\gamma,
\end{equation*}
so in summary we have obtained that for any $s$ with $c_s \neq 0$, $[s,z_k^1w]=[h, z_k^1w]$ holds if and only if $s=g$ with $\pi_1(g)=\pi_1(h)$. Then combining with \eqref{eq:nonzero2}, we deduce
\begin{equation*}
  \ff\left(\left[h, z_k^1w\right]\right) = \sum_{\pi_1(g)=\pi_1(h)} c_g \neq 0,
\end{equation*}
which shows that the open set $(h, D(z^1_k))$ is contained in $\supp \ff$.
\end{proof}

\bibliography{bibgrpdsing}

\end{document}